\documentclass[conference]{IEEEtran}
\IEEEoverridecommandlockouts
\usepackage{cite}
\usepackage{amsmath,amssymb,amsfonts}
\usepackage{graphicx}
\usepackage{textcomp}
\usepackage{xcolor}

\usepackage{algorithm}
\usepackage{algpseudocode}
\usepackage{amsthm}
\usepackage{array}
\usepackage{booktabs}
\usepackage{cleveref}
\usepackage{color, colortbl}
\usepackage{comment}
\usepackage{mathtools}
\usepackage{multirow}
\usepackage{subcaption}
\usepackage{tikz}
\usepackage{url}

\newtheorem{definition}{Definition}
\newtheorem{theorem}{Theorem}
\newtheorem{example}{Example}
\newtheorem{remark}{Remark}

\crefname{table}{Table}{Tables}
\crefname{figure}{Fig.}{Fig.}
\crefname{equation}{}{}
\crefname{definition}{Definition}{Definitions}
\crefname{algorithm}{Algorithm}{Algorithms}

\def\BibTeX{{\rm B\kern-.05em{\sc i\kern-.025em b}\kern-.08em
    T\kern-.1667em\lower.7ex\hbox{E}\kern-.125emX}}
\begin{document}

\title{Solving Linear Systems on  a GPU with Hierarchically Off-Diagonal Low-Rank Approximations
}

\author{\IEEEauthorblockN{Chao Chen}
\IEEEauthorblockA{\textit{Oden Institute for Computational Engineering and Sciences} \\
\textit{University of Texas at Austin}\\
Austin, United States \\
chenchao.nk@gmail.com}
\and
\IEEEauthorblockN{Per-Gunnar Martinsson}
\IEEEauthorblockA{\textit{Department of Mathematics} \\
\textit{Oden Institute for Computational Engineering and Sciences} \\
\textit{University of Texas at Austin}\\
Austin, United States  \\
pgm@oden.utexas.edu}
}

\maketitle

\newcommand{\A}{\mathsf{A}}
\newcommand{\x}{\mathsf{x}}

\newcommand{\rchol}{\texttt{rchol}}
\newcommand{\ichol}{\texttt{ichol}}

\newcommand{\hlu}{$\mathcal{H}$-LU }
\newcommand{\hodlrlib}{{\textsc{HODLRlib}}}

\newcommand{\red}[1]{\textcolor{red}{#1}}

\newcommand{\N}{ {\mathcal{N}} }
\newcommand{\G}{ {\mathcal{G}} }
\newcommand{\bigO}{ {\mathcal{O}} }

\renewcommand{\algorithmiccomment}[1]{#1}

\renewcommand{\algorithmicrequire}{\textbf{Input:}}
\renewcommand{\algorithmicensure}{\textbf{Output:}}

\newcolumntype{H}{>{\setbox0=\hbox\bgroup}c<{\egroup}@{}}
	
\definecolor{Gray}{gray}{0.9}

\newcommand{\nomen}[1]{{\textcolor{magenta}{#1}}}

\graphicspath{{figs/}}

\begin{abstract}
We are interested in solving linear systems arising from three applications: (1) kernel methods in machine learning, (2) discretization of boundary integral equations from mathematical physics, and (3) Schur complements formed in the factorization of many large sparse matrices. The coefficient matrices are often data-sparse in the sense that their off-diagonal blocks have low numerical ranks; specifically, we focus on ``hierarchically off-diagonal low-rank (HODLR)'' matrices. We introduce algorithms for factorizing HODLR matrices and for applying the factorizations on a GPU. The algorithms leverage the efficiency of batched dense linear algebra, and they scale nearly linearly with the matrix size when the numerical ranks are fixed. The accuracy of the HODLR-matrix approximation is a tunable parameter, so we can construct high-accuracy fast direct solvers or low-accuracy robust preconditioners. Numerical results show that we can solve problems with several millions of unknowns in a couple of seconds on a single GPU.
\end{abstract}


\begin{IEEEkeywords}
Linear solver on GPU, boundary integral equation, kernel matrix, elliptic partial differential equations, hierarchical low-rank approximation, batched dense linear algebra, rank structured matrix, hierarchical matrix, LU factorization.
\end{IEEEkeywords}

\section{Introduction}
Consider the solution of a large {dense} linear system
\begin{equation} \label{e:axb}
A \, x = b, \quad A \in \mathbb{F}^{N\times N}, {x \text{ and } b} \in \mathbb{F}^N,
\end{equation}
on a GPU, where $\mathbb{F}$ is the field of real or complex numbers.
We focus on the situation where the coefficient matrix $A$ can be approximated by \emph{hierarchically off-diagonal low-rank} (HODLR) matrices~\cite{ambikasaran2013mathcal,aminfar2016fast}.
Given a matrix $X$ partitioned into a $2 \times 2$ block form:
\begin{equation} \label{e:hodlr}
X =
\begin{pmatrix}
X_{11} & X_{12} \\
X_{21} & X_{22}
\end{pmatrix}
,
\end{equation}
we say $X$ is a HODLR matrix if (1) the two off-diagonal blocks $X_{12}$ and $X_{21}$ are low rank, and (2) the two diagonal blocks $X_{11}$ and $X_{22}$ have the same off-diagonal low-rank structure or have sufficiently small sizes. (See a pictorial illustration in \cref{f:hodlr}.)

Solving \cref{e:axb} on a GPU using the HODLR-matrix approximation is particularly appealing for two reasons. (1) The approximation reduces the required memory footprint and thus allows solving much larger problem sizes than storing the entire matrix $A$. For example, we were able to solve problems with several millions of unknowns on a single GPU that has only 32 GB of memory (see, e.g., \cref{t:bie_dr}). (2) Our algorithms are built upon the batched matrix-matrix multiplication and the batched LU factorization routines, which are highly efficient on GPUs. For example, the construction of our GPU solver achieved approximately 2 TFlop/s on an NVIDIA V100 GPU (see \cref{f:flops}).

HODLR matrices arise from a range of applications across science,
engineering, and data analytics, including:

\paragraph{Kernel matrices}

Given a (real) kernel function $\mathcal{K}$, such as the Gaussian kernel and the Matern kernel,  and a data set $\{{y}_i\}_{i=1}^N$, the associated kernel matrix ${K}  \in \mathbb{R}^{N\times N}$ is defined as
\[
{K}_{i,j} = \mathcal{K}({y}_i,{y}_j), \quad\quad \forall i,j = 1,2,\ldots,N.
\]
Such matrices arise in machine learning~\cite{gray2001n,hofmann2008kernel} and data assimilation~\cite{ambikasaran2013fast,li2014kalman}. Ambikasaran et al.~\cite{ambikasaran2014fast,ambikasaran2015fast} demonstrated that these matrices can be approximated efficiently by HODLR matrices. See also our numerical results in \cref{s:kernel}.

\paragraph{Boundary integral equations}

Some boundary value problems (BVPs) involving, {e.g.}, the Laplace and the Helmholtz equations, can be reformulated as boundary integral equations (BIEs) of the form
\begin{equation} \label{e:ie}
a(x) u(x) + \int_{\Gamma} K(x,y) u(y) dy = f(x)
\end{equation}
where $K(\cdot, \cdot)$ is derived from the free-space fundamental solution associated with the elliptic operator; where $a(\cdot)$ and $f(\cdot)$ are given functions; and where $u(\cdot)$ is the unknown. While it may be challenging to discretize the original partial differential equations (PDEs) directly, the BIEs offer several advantages (see, e.g., \cite[Chapter~10]{martinsson2019fast}). Although the discretization of \cref{e:ie} leads to a dense linear system, the discretized integral operator can be approximated efficiently by a HODLR matrix in many environments. See \cref{s:laplace,s:helmholtz} for two concrete examples.

\paragraph{Elliptic PDEs}

The discretization of an  elliptic PDE
\[
- \nabla \cdot (a(x) \nabla u(x)) + b(x) u(x) = f(x),
\]
where $a(\cdot)$, $b(\cdot)$, and $f(\cdot)$ are given functions, leads to a \emph{sparse} linear system to be solved.
While sparse direct solvers~\cite{davis2016survey} based on Gaussian elimination are extremely robust, they typically
require significant computing resources to handle dense Schur complements arising during the elimination.
That significant acceleration can be attained by exploiting rank-structures in these Schur complements was established
in, e.g., \cite{aminfar2016fast,xia2010superfast,10.1007/s10915-008-9240-6}.

\subsection{Previous work}

Classical direct methods for solving \cref{e:axb} based on LU or QR factorizations admit highly efficient implementations on GPUs, but are limited by their unfavorable $\bigO(N^3)$ scaling in terms of operations. For GPU computing, the $\bigO(N^2)$ storage complexity can be even more restrictive.

This paper concerns a class of methods that exploit the numerically low-rank property of off-diagonal blocks in the matrix $A$. As many theoretical and empirical results have demonstrated, certain off-diagonal blocks in $A$ can be compressed efficiently by their low-rank approximations. Based on this observation, the tile low-rank (TLR) approximation~\cite{ltaief2021meeting,amestoy2019performance} views matrix $A$ as a block matrix and compresses off-diagonal blocks using low-rank approximations. With the TLR approximation, the LU factorization or the Cholesky factorization of matrix $A$ can be accelerated significantly. This approach has been realized efficiently on shared and distributed-memory systems for covariance matrices~\cite{akbudak2017tile} and BIEs~\cite{al2020solving}. It is also implemented on GPUs recently~\cite{boukaram2021h2opus}. However, the asymptotic complexity for solving \cref{e:axb} is generally super-linear and the cost can be significant when the matrix size $N$ is large.

To arrive at linear or nearly linear complexities, a hierarchical decomposition of matrix $A$ is needed, and several hierarchically low-rank approximations have been proposed including $\mathcal{H}$ matrices~\cite{hackbusch1999sparse,5555}, $\mathcal{H}^2$ matrices~\cite{hackbusch2002data}, hierarchically semi-separable (HSS) matrices~\cite{martinsson2005fast,chandrasekaran2006fast,xia2010fast} and HODLR matrices. See related algorithms for solving \cref{e:axb} in, e.g., ~\cite{grasedyck2009domain,kriemann2013fancyscript,ambikasaran2013fast,ambikasaran2015fast,ho2016hierarchical,rouet2016distributed,coulier2017inverse,minden2017recursive,martinsson2019fast,xia2021multi,liu2021sparse,sushnikova2022fmm} and the references therein. However, few of these methods have been implemented to achieve high performance on modern computing architectures. Two recent software packages for HODLR matrices include (1)  \hodlrlib~\cite{ambikasaran2019hodlrlib}, which is an open-source code written in C++ and parallelized with OpenMP for multicore CPUs, and (2)  hm-toolbox~\cite{massei2020hm}, which is a MATLAB code that implements many operations involving HODLR matrices focusing on prototyping algorithms and ensuring reproducibility rather than delivering high performance.



\subsection{Contributions}

We introduce a new data structure for HODLR matrices, where the low-rank bases of all off-diagonal blocks are concatenated into two big matrices (see an example in \cref{f:cat}). This strategy aggregates small sub-blocks in the bases that need to be processed during the factorization into large sub-blocks in the big matrices. As a result, we can combine sequences of BLAS/LAPACK calls of low arithmetic intensities into a single kernel launch, which increases the Flop rate and mitigates the overhead of data transfer. This approach naturally leads to (1) a high-performance factorization algorithm that delivered up to 20 GFlop/s on a single CPU core in our experiments, and (2) parallel algorithms for factorizing a HODLR matrix and for applying the factorization to solve linear systems. Furthermore, we implemented the parallel algorithms leveraging existing high-performance  batched matrix-matrix multiplication and  batched LU factorization routines on a GPU. We present numerical benchmarks for solving large dense linear systems arising from kernel methods and two dimensional BIEs, and we compare the performance of our methods to existing methods/codes.

%
%
%
%
%
%

\section{Preliminariles}

We adopt the following MATLAB notations: (1) $A(\mathcal{I}, :)$ and $A(:, \mathcal{I})$ denote the rows and columns in matrix $A$ corresponding to an index set $\mathcal{I}$, respectively; and (2) $[ \, A  \, | \,  B \, ]$ denotes the concatenation of two matrices $A$ and $B$ that have the same number of rows.
Matrices and vectors are denoted using upper and lower case letters, respectively.
In particular, $I$ is used to denote an identity matrix of an appropriate size.
We use greek letters to denote nodes in a tree data structure.

\subsection{HODLR matrix}

Here we give  (non-recursive) definitions of a cluster tree and the associated HODLR format, which follow closely with~\cite{massei2020hm, martinsson2019fast}.
A cluster tree stands for a hierarchical partitioning of the row/column indices of a matrix, which dictates a tessellation of the matrix  in a HODLR format.

\begin{definition}[Cluster tree] \label{d:tree}
Given an index set $\mathcal{I} \coloneqq \{1,2,\ldots,N\}$, a {cluster tree} $\mathcal{T}_L$ is a binary tree that satisfies the following three conditions:
\begin{enumerate}
\item
There are $L+1$ levels, namely, $0,1,\ldots,L$, and there are $2^{\ell}$ nodes at level $\ell$.

\item
Every tree node  stands for a (nonempty) consecutive subset of $\mathcal{I}$. In particular, the root node represents $\mathcal{I}$.

\item
The union of two subsets owned by a pair of siblings, respectively, equals to the subset owned by their parent. (Nodes at the same level form a partitioning of $\mathcal{I}$.)

\end{enumerate}
\end{definition}

\cref{f:tree} shows an example of {$\mathcal{T}_2$}. In practice, the indices in $\mathcal{I}$ are usually associated with points in {$\mathbb{R}^k$}, so $\mathcal{T}_L$ can be computed using some recursive bisection strategies. For example, $\mathcal{T}_L$ can be constructed as a $k$-d tree.
Notice that a cluster tree $\mathcal{T}_L$ naturally leads to a tessellation of a matrix $A \in \mathbb{F}^{N\times N}$:
\begin{enumerate}
\item
 A leaf node $\alpha$ corresponds to a diagonal block $A(\mathcal{I}_{\alpha}, \mathcal{I}_{\alpha})$, and there are $2^L$ of them.

\item
A pair of siblings $\alpha$ and $\beta$ corresponds to an off-diagonal block $A(\mathcal{I}_{\alpha}, \mathcal{I}_{\beta})$, and there are $2^{L+1} - 2$ of them.
\end{enumerate}
For example, \cref{f:hodlr} illustrates such a tessellation.


\begin{figure}
\centering
\begin{tikzpicture}[level/.style={sibling distance=30mm/#1, level distance = 1.1cm}]
\node [circle,draw] (z){$1$}
  child {node [circle,draw] (a) {$2$}
    child {node [circle,draw] (b) {$4$}
        child [grow=left] {node (q) [xshift = -0.8cm] {level 2} edge from parent[draw=none]
          child [grow=up] {node (r) {level 1} edge from parent[draw=none]
            child [grow=up] {node (s) {level 0} edge from parent[draw=none]
            }
          }
        }
    }
    child {node [circle,draw] (g) {$5$}
    }
  }
  child {node [circle,draw] (j) {$3$}
    child {node [circle,draw] (k) {$6$}
    }
    child {node [circle,draw] (l) {$7$}
    }
};
\node at (1.2,0.1)   {\scriptsize $\mathcal{I}_1=1:400$};
\node at (2.5,-0.6)   {\scriptsize $\mathcal{I}_3=201:400$};
\node at (3.1,-1.7)   {\scriptsize $\mathcal{I}_7=301:400$};
\node at (-1.8,-0.6)   {\scriptsize $\mathcal{I}_2=1:200$};
\node at (0.,-1.7)   {\scriptsize $\mathcal{I}_5=101:200$};
\node at (-2.7,-1.7)   {\scriptsize $\mathcal{I}_4=1:100$};
\end{tikzpicture}
\caption{An example of hierarchical decomposition of matrix row/column indices $\mathcal{I}=\{1,2,\ldots,400\}$. The root is the entire set $\mathcal{I}_1=\mathcal{I}$, and there are four leaf nodes in this case. Some of the individual index vectors are given. Node 2 has two ``children'' 4 and 5, who are ``siblings''.}
\label{f:tree}
\end{figure}
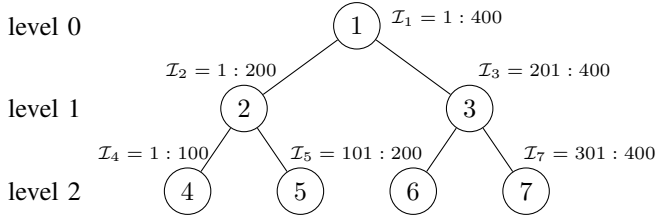

\begin{figure}
\centering
\begin{tikzpicture}
\node at (-0.7, 2)   {$A =$};
\draw[very thick] (0,0) rectangle (4,4);
\draw[gray, thick] (2,0) -- (2,4);
\draw[gray, thick] (0,2) -- (4,2);
\draw[gray] (3,0) -- (3,2);
\draw[gray] (1,2) -- (1,4);
\draw[gray] (0,3) -- (2,3);
\draw[gray] (2,1) -- (4,1);
\node at (1,1)   {$A_{3,2}$};
\node at (3,3)   {$A_{2,3}$};
\node at (0.5,2.5)   {$A_{5,4}$};
\node at (1.5,3.5)   {$A_{4,5}$};
\node at (2.5,0.5)   {$A_{7,6}$};
\node at (3.5,1.5)   {$A_{6,7}$};
\node at (0.5,3.5)   {$D_{4}$};
\node at (1.5,2.5)   {$D_{5}$};
\node at (2.5,1.5)   {$D_{6}$};
\node at (3.5,0.5)   {$D_{7}$};
\end{tikzpicture}
\caption{Matrix tessellation corresponding to the cluster tree in \cref{f:tree}. A diagonal block $D_{{\alpha}}$ is defined as $A(\mathcal{I}_{\alpha}, \mathcal{I}_{\alpha})$, where ${\alpha}$ is a leaf node.}
\label{f:hodlr}
\end{figure}


\begin{definition}[HODLR matrix]
Given  a {cluster tree} $\mathcal{T}_L$,  a matrix $A \in \mathbb{F}^{N\times N}$ is a HODLR matrix if every off-diagonal block $A(\mathcal{I}_{\alpha}, \mathcal{I}_{\beta})$ that corresponds to a pair of siblings $\alpha$ and $\beta$ in $\mathcal{T}_L$ has low rank.
We say a HODLR matrix has rank $r$ if the maximum rank of all off-diagonal blocks is $r$.
\end{definition}


\subsection{Construction of a HODLR matrix}

The construction of a HODLR matrix is straightforward in the sense that we need to only compress certain off-diagonal blocks in the original matrix. We refer interested readers to Ambikasaran's PhD thesis~\cite{ambikasaran2013fast} for a review of algebraic and analytic techniques to compute the low-rank approximations. See also~\cite{dong2021simpler} for some recent advances on randomized methods.
In situations where a fast matrix-vector product routine exists for the matrix to be compressed, the so-called ``peeling algorithms''~\cite{lin2011fast,martinsson2016compressing,levitt2022linear} have been developed for the construction of a HODLR approximation. 

Several parallel algorithms have been developed to construct more complicated formats than HODLR matrices, and some of them can be used to construct a HODLR-matrix approximation. Fernando et al~\cite{fernando2017scalable} demonstrates the construction of an HSS matrix on multiple GPUs using an almost matrix-free method introduced in~\cite{martinsson2011fast}. Boukaram et al~\cite{boukaram2019randomized} introduces a matrix-free method to construct an $\mathcal{H}^2$ matrix on a GPU. Chenhan et al~\cite{yu2017geometry,chenhan2018distributed} develop algorithms on multicore and distributed-memory machines to approximate a symmetric positive definite (SPD) matrix with hierarchically low-rank structures.

\section{Algorithms}

Assume a rank-$r$ HODLR matrix $A \in \mathbb{F}^{N\times N}$ and the underlying cluster tree $\mathcal{T}_L$ are given. Our focus here is factorizing the HODLR matrix and applying the factorization to solve a linear system
\begin{equation} \label{e:hxb}
A\, x = b.
\end{equation}
Let every off-diagonal block $A(\mathcal{I}_{\alpha}, \mathcal{I}_{\beta})$ (that corresponds to a pair of siblings in $\mathcal{T}_L$) be that
\begin{equation} \label{e:rank}
A(\mathcal{I}_{\alpha}, \mathcal{I}_{\beta}) = U_{\alpha} V_{\beta}^*,
\end{equation}
where $U_{\alpha}$ and $V_{\beta}$ are two skinny matrices, and we call them the left and the right low-rank bases.

\subsection{Recursive algorithm}

Consider a partitioning of \cref{e:hxb} into
\begin{equation} \label{e:axb_blk}
\begin{pmatrix}
A_{\alpha} & U_{\alpha} V_{\beta}^* \\
U_{\beta} V_{\alpha}^* & A_{\beta}
\end{pmatrix}
\begin{pmatrix}
x_{\alpha} \\
x_{\beta}
\end{pmatrix}
=
\begin{pmatrix}
b_{\alpha} \\
b_{\beta}
\end{pmatrix},
\end{equation}
where $\alpha$ and $\beta$ are the two nodes at level 1 in $\mathcal{T}_L$. Suppose we can solve the following two subproblems with multiple right-hand sides:
%
\begin{equation} \label{e:recursion}
\left\{
\begin{array}{c}
A_{\alpha} \, z_{\alpha} = b_{\alpha} \\
A_{\alpha} \, Y_{\alpha} = U_{\alpha}
\end{array}
\right.
\quad
\text{and}
\quad
\left\{
\begin{array}{c}
A_{\beta} \, z_{\beta} = b_{\beta} \\
A_{\beta} \, Y_{\beta} = U_{\beta}
\end{array}
\right.
.
\end{equation}
Then, the  solution of \cref{e:hxb} can be obtained via
\begin{equation} \label{e:solution}
\begin{pmatrix}
x_{\alpha} \\
x_{\beta}
\end{pmatrix}
=
\begin{pmatrix}
z_{\alpha} \\
z_{\beta}
\end{pmatrix}
-
\begin{pmatrix}
Y_{\alpha} w_{\alpha} \\
Y_{\beta} w_{\beta}
\end{pmatrix}
,
\end{equation}
where $w_{\alpha}$ and $w_{\beta}$ are from solving 
\begin{equation} \label{e:node}
\begin{pmatrix}
V_{\alpha}^* Y_{\alpha} & I \\
I & V_{\beta}^* Y_{\beta}
\end{pmatrix}
\begin{pmatrix}
w_{\alpha} \\
w_{\beta}
\end{pmatrix}
=
\begin{pmatrix}
V_{\alpha}^* z_{\alpha} \\
V_{\beta}^* z_{\beta}
\end{pmatrix}
.
\end{equation}
Obviously, \cref{e:recursion} can be solved recursively when $\alpha$ and $\beta$ are not leaves in $\mathcal{T}_L$; otherwise, we solve \cref{e:recursion} directly using, e.g., the LU factorization. This recursive algorithm initially appeared in~\cite{aminfar2016fast}, and we prove its correctness below.

\begin{theorem} [Correctness of recursion]
\label{t:recursion}
Assume the following three matrices are invertible:
\[
A_{\alpha}, \quad A_{\beta}, {\text \quad and \quad}
\begin{pmatrix}
V_{\alpha}^* Y_{\alpha} & I \\
I & V_{\beta}^* Y_{\beta}
\end{pmatrix}
.
\]
Then, \cref{e:solution} is the solution of \cref{e:axb_blk}.
\end{theorem}

\begin{proof}
Multiply
$\begin{pmatrix}
A_{\alpha}^{-1} &  \\
 & A_{\beta}^{-1}
\end{pmatrix}$
on both sizes of \cref{e:axb_blk}, and it is sufficient to prove
\[
\begin{pmatrix}
I & Y_{\alpha} V_{\beta}^* \\
Y_{\beta} V_{\alpha}^* & I
\end{pmatrix}
\begin{pmatrix}
x_{\alpha} \\
x_{\beta}
\end{pmatrix}
=
\begin{pmatrix}
z_{\alpha} \\
z_{\beta}
\end{pmatrix}.
\]
According to the Woodbury formula, we know that
\begin{align} \label{e:wood}
&
\begin{pmatrix}
I & Y_{\alpha} V_{\beta}^* \\
Y_{\beta} V_{\alpha}^* & I
\end{pmatrix}
^{-1}
=
\left[
I +
\begin{pmatrix}
 & Y_{\alpha}  \\
Y_{\beta} &
\end{pmatrix}
\begin{pmatrix}
V_{\alpha}^* & \\
& V_{\beta}^*
\end{pmatrix}
\right
]^{-1}
\notag \\
&= I -
\begin{pmatrix}
 & Y_{\alpha}  \\
Y_{\beta} &
\end{pmatrix}
\begin{pmatrix}
I & V_{\alpha}^* Y_{\alpha}  \\
V_{\beta}^* Y_{\beta} & I
\end{pmatrix}
^{-1}
\begin{pmatrix}
V_{\alpha}^* & \\
& V_{\beta}^*
\end{pmatrix}
.
\end{align}
Therefore, it is straightforward to verify that \cref{e:solution} holds.
\end{proof}

Observe that some calculations in the recursion do not depend on the right-hand side $b$, and thus intermediate results can be precomputed. In particular, we divide the computation into two stages:
\begin{itemize}
\item
Factorization: we compute $Y_{\alpha}$, $Y_{\beta}$, and an LU factorization of
\begin{equation} \label{e:k}
K_{\gamma} \triangleq \begin{pmatrix}
V_{\alpha}^* Y_{\alpha} & I \\
I & V_{\beta}^* Y_{\beta}
\end{pmatrix},
\end{equation}
where $\gamma$ is the parent of $\alpha$ and $\beta$.

\item
Solution: we compute $z_{\alpha}$, $z_{\beta}$, $w_{\alpha}$, $w_{\beta}$, $x_{\alpha}$, and $x_{\beta}$.
\end{itemize}
Notice that some of the calculations can be done in-place (overwriting  inputs with outputs). For example, we can overwrite two U matrices with corresponding Y matrices. As a result, the extra memory required by the algorithm mostly comes from storing the LU factorization of the coefficient matrix and the right-hand sides in \cref{e:node}, which is negligible when the rank in \cref{e:rank} is small.
%

\subsection{Data structure and nonrecursive algorithm}

For simplicity, let us assume the rank of all off-diagonal blocks is $r$. (The following presentation extends to general cases where the ranks vary.) To motivate the new data structure we use in our parallel algorithms, we illustrate the main idea through the following example.

\begin{example} 
Consider the HODLR matrix in \cref{f:hodlr}, where the underlying cluster tree has  three levels (level 0, 1, and 2).
An appropriate partitioning of \cref{e:hxb} leads to
\[
\begin{pmatrix}
A_{2} & U_{2} V_{3}^* \\
U_{3} V_{2}^* & A_{3}
\end{pmatrix}
\begin{pmatrix}
x_{2} \\
x_{3}
\end{pmatrix}
=
\begin{pmatrix}
b_{2} \\
b_{3}
\end{pmatrix},
\]
where $A_{2}$ and $A_{3}$ are themselves (2-level) HODLR matrices; $b_{2} = b(\mathcal{I}_2)$ and $b_{3} = b(\mathcal{I}_3)$ are subvectors in the right-hand side $b$; and $x_{2} = x(\mathcal{I}_2)$ and $x_{3} = x(\mathcal{I}_3)$ are subvectors in the solution $x$.

Let us write down the recursion steps.
At the first step, we have two subproblems:
\begin{equation*}
\left\{
\begin{array}{c}
A_{2} \, z_{2} = b_{2} \\
A_{2} \, Y_{2} = U_{2}
\end{array}
\right.
\quad
\text{and}
\quad
\left\{
\begin{array}{c}
A_{3} \, z_{3} = b_{3} \\
A_{3} \, Y_{3} = U_{3}
\end{array}
\right.
.
\end{equation*}
Rewrite them more compactly as
\begin{equation*}
A_{2} \, [  \, z_{2}  \, |  \, Y_{2}  \, ] = [  \, b_{2}  \, |  \, U_{2}  \, ]
\quad
\text{and}
\quad
A_{3} \, [ \, z_{3} \,  |  \, Y_{3}  \, ] = [  \, b_{3}  \, |  \, U_{3}  \, ],
\end{equation*}
which again can be partitioned appropriately into
\begin{align*}
&
\begin{pmatrix}
D_{4} & U_{4} V_{5}^* \\
U_{5} V_{4}^* & D_{5}
\end{pmatrix}
\left(
\begin{array}{c|c}
z_{4} & Y_{2}^{\text{top}}  \\
z_{5} & Y_{2}^{\text{bot}}
\end{array}
\right)
=
\left(
\begin{array}{c|c}
b_{4} & U_{2}^{\text{top}}  \\
b_{5} & U_{2}^{\text{bot}}
\end{array}
\right)
\\
\text{and} \quad
&
\begin{pmatrix}
D_{6} & U_{6} V_{7}^* \\
U_{7} V_{6}^* & D_{7}
\end{pmatrix}
\left(
\begin{array}{c|c}
z_{6} & Y_{3}^{\text{top}}  \\
z_{7} & Y_{3}^{\text{bot}}
\end{array}
\right)
=
\left(
\begin{array}{c|c}
b_{6} & U_{3}^{\text{top}}  \\
b_{7} & U_{3}^{\text{bot}}
\end{array}
\right)
,
\end{align*}
where $b_{4} = b(\mathcal{I}_4)$, $b_{5} = b(\mathcal{I}_5)$, $b_{6} = b(\mathcal{I}_6)$, and $b_{7} = b(\mathcal{I}_7)$ are subvectors in the right-hand side $b$ (notice that $\mathcal{I}_4 \cup \mathcal{I}_5 = \mathcal{I}_2$ and $\mathcal{I}_6 \cup \mathcal{I}_7 = \mathcal{I}_3$ according to the cluster tree in \cref{f:tree}).

At the second step, we have four subproblems:
\begin{align*}
&
\left\{
\begin{array}{c}
D_{4} \, [ \, z_{4}  \, |  \, Y_{2}^{\text{top}}  \, |  \, Y_{4}  \, ] = [ \, b_{4}  \, | \,  U_{2}^{\text{top}}  \, | \,  U_{4}  \,] \\
D_{5} \, [ \, z_{5}  \, |  \, Y_{2}^{\text{bot}}  \, |  \, Y_{5}  \, ] = [ \, b_{5}  \, | \,  U_{2}^{\text{bot}}  \, | \,  U_{5}  \,]
\end{array}
\right.
\\
\text{and}
\quad
&
\left\{
\begin{array}{c}
D_{6} \, [ \, z_{6}  \, |  \, Y_{3}^{\text{top}}  \, |  \, Y_{6}  \, ] = [ \, b_{6}  \, | \,  U_{3}^{\text{top}}  \, | \,  U_{6}  \,] \\
D_{7} \, [ \, z_{7}  \, |  \, Y_{3}^{\text{bot}}  \, |  \, Y_{7}  \, ] = [ \, b_{7}  \, | \,  U_{3}^{\text{bot}}  \, | \,  U_{7}  \,]
\end{array}
\right.
,
\end{align*}
where the four $D$ matrices are corresponding diagonal blocks (that have full ranks), and we solve them directly using, e.g., the LU factorization.

Observe the right-hand sides in the above four subproblems, where the $U$ matrices have been partitioned and then concatenated together. This motivates the idea of concatenating all the  $U$ bases into one big matrix $U_{\text{big}}$ as shown in \cref{f:cat}. The same strategy is applied to form a big matrix $V_{\text{big}}$ of all the $V$ bases for reasons that will be clear when we introduce our GPU algorithms. We also concatenate diagonal blocks and K matrices defined at \cref{e:k}  for ease of presentation.

\end{example}

\begin{figure}
\centering
\begin{tikzpicture}
\node at (-0.7, 2)   {$U_{\text{big}} =$};
\draw[very thick] (0,0) rectangle (1.7,4);
\draw[gray, thick] (1,0) -- (1,4);
\draw[gray] (0,2) -- (1.7,2);
\draw[gray] (1,1) -- (1.7,1);
\draw[gray] (1,3) -- (1.7,3);
\node at (0.5, 3)   {$U_2$};
\node at (0.5, 1)   {$U_3$};
\node at (1.35, 0.5)   {$U_7$};
\node at (1.35, 1.5)   {$U_6$};
\node at (1.35, 2.5)   {$U_5$};
\node at (1.35, 3.5)   {$U_4$};
\end{tikzpicture}
\hspace{1cm}
\begin{tikzpicture}
\node at (-0.7, 2)   {$V_{\text{big}} =$};
\draw[very thick] (0,0) rectangle (1.7,4);
\draw[gray, thick] (1,0) -- (1,4);
\draw[gray] (0,2) -- (1.7,2);
\draw[gray] (1,1) -- (1.7,1);
\draw[gray] (1,3) -- (1.7,3);
\node at (0.5, 3)   {$V_2$};
\node at (0.5, 1)   {$V_3$};
\node at (1.35, 0.5)   {$V_7$};
\node at (1.35, 1.5)   {$V_6$};
\node at (1.35, 2.5)   {$V_5$};
\node at (1.35, 3.5)   {$V_4$};
\end{tikzpicture}
\caption{Concatenation of all low-rank bases into two big matrices for the HODLR matrix in \cref{f:hodlr} assuming the ranks at the same level are the same.}
\label{f:cat}
\end{figure}
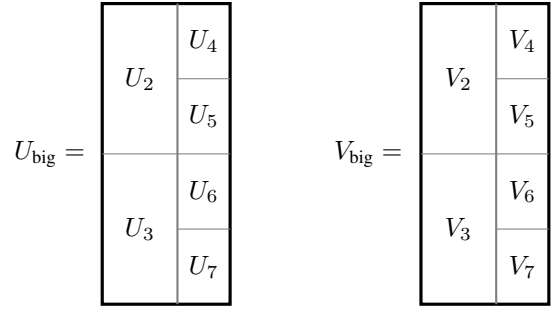

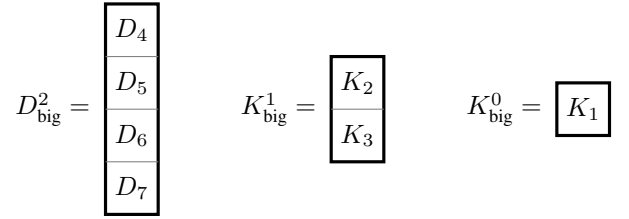
\begin{figure}
\centering
\begin{tikzpicture}
\node at (-0.7, 1.4)   {$D_{\text{big}}^2 =$};
\draw[very thick] (0,0) rectangle (0.7, 2.8);
\draw[gray] (0, 0.7) -- (0.7, 0.7);
\draw[gray] (0, 1.4) -- (0.7, 1.4);
\draw[gray] (0, 2.1) -- (0.7, 2.1);
\node at (0.35, 0.35)   {$D_7$};
\node at (0.35, 1.05)   {$D_6$};
\node at (0.35, 1.75)   {$D_5$};
\node at (0.35, 2.45)   {$D_4$};
\node at (2.3, 1.4)   {$K_{\text{big}}^1 =$};
\draw[very thick] (3,0.7) rectangle (3.7, 2.1);
\draw[gray] (3, 1.4) -- (3.7, 1.4);
\node at (3.35, 1.05)   {$K_3$};
\node at (3.35, 1.75)   {$K_2$};
\node at (5.3, 1.4)   {$K_{\text{big}}^0 =$};
\draw[very thick] (6,1.05) rectangle (6.7, 1.75);
\node at (6.35, 1.4)   {$K_1$};
\end{tikzpicture}
\caption{Concatenation of diagonal blocks and K matrices defined at \cref{e:k} for the HODLR matrix in \cref{f:hodlr} assuming the ranks at the same level are the same. A superscript $\ell$ means a block-row view of the matrix, where matrix rows are partitioned according to nodes at the $\ell$-th level in the tree.}
\label{f:dk}
\end{figure}


In general, we traverse $\mathcal{T}_L$ in a breadth-first (top-down level-by-level) order to create the two big matrices. Consider every node $\gamma$ that has two children $\alpha$ and ${\beta}$ (so $\alpha$ and ${\beta}$ are siblings). First, we place $U_{\alpha}$ and $U_{\beta}$ vertically on top of each other as
$\begin{pmatrix}
U_{\alpha} \\
U_{\beta}
\end{pmatrix}$ (the order does not matter). In general, $U_{\alpha}$ and $U_{\beta}$ may not have the same number of columns, and we align them to the left. We do the same thing with $V_{\alpha}$ and $V_{\beta}$.
Second, if $\gamma$ is not the root of $\mathcal{T}_L$, it is associated with low-rank bases $U_{\gamma}$ and $V_{\gamma}$, and we form the following concatenation
\begin{equation} \label{e:u}
\left[
\begin{array}{c|c}
U_{\gamma}
&
\begin{array}{c}
U_{\alpha} \\
U_{\beta}
\end{array}
\end{array}
\right]
\quad \text{and} \quad
\left[
\begin{array}{c|c}
V_{\gamma}
&
\begin{array}{c}
V_{\alpha} \\
V_{\beta}
\end{array}
\end{array}
\right]
.
\end{equation}
Notice that the number of rows in $U_{\gamma}$ equals to the sum of those in $U_{\alpha}$ and $U_{\beta}$ and the same applies to the $V$ matrices. Finally, we obtain two big matrices $U_{\text{big}}$ and $V_{\text{big}}$ after  traversing $\mathcal{T}_L$. 
With the previous assumption of constant rank, $U_{\text{big}}$ and $V_{\text{big}}$ are two matrices of size $N$-by-$rL$. In general, when ranks of off-diagonal blocks vary, $U_{\text{big}}$ and $V_{\text{big}}$ are no longer rectangular matrices. But the algorithms we are going to introduce can be generalized accordingly. In the same spirit, we concatenate the D matrices and K matrices as shown in \cref{f:dk}.

Given a  cluster tree $\mathcal{T}_L$ as defined in \cref{d:tree}, we can unroll the previous recursive algorithm into a for-loop based algorithm, which consists of a factorization stage and a solution stage as described in \cref{a:fact,a:solve}, respectively. The advantages of the new data structure are clear. We need only one BLAS or LAPACK call for the computation involving multiple left low-rank bases across different levels in $\mathcal{T}_L$, and there is no unnecessary data movement. For example, one LAPACK call (getrs) is sufficient for solving all the right-hand sides at a leaf node in $\mathcal{T}_L$ (line 4 in \cref{a:fact}).

Notice that a lot of memory allocation in \cref{a:fact} is avoided by overwriting inputs with outputs. For example, $Y_{\text{big}}$ overwrites $U_{\text{big}}$, and the factorizations of diagonal blocks are stored in place. As a result, \cref{a:fact} requires little extra memory when the rank $r$ is small.

\begin{algorithm}
\caption{Nonrecursive algorithm for solving \cref{e:hxb}: factorization stage}
\label{a:fact}
\begin{algorithmic}[1]
\Require diagonal blocks in $A$, matrix $U_{\text{big}}$, matrix $V_{\text{big}}$, and cluster tree $\mathcal{T}_L$.
\Ensure matrix $Y_{\text{big}}$ and stored factorizations.
%
\State $Y_{\text{big}} \gets U_{\text{big}}$
\hfill \Comment \emph{// $Y_{\text{big}}$ overwrites $U_{\text{big}}$.}
%
\For {leaf node $\alpha$ in $\mathcal{T}_L$}
\State Factorize diagonal block $D_{\alpha} = A(\mathcal{I}_{\alpha}, \mathcal{I}_{\alpha})$ and \underline{store its LU factorization in-place}.
\State Apply $D_{\alpha}^{-1}$ (the previous LU factorization) to solve multiple right-hand sides $Y_{\text{big}}(\mathcal{I}_{\alpha},:)$ {in-place}.
\EndFor
%
\For {level $\ell=L-1$ \textbf{to} $0$}
\For {tree node $\gamma$ at level $\ell$ in $\mathcal{T}_L$}
\State Let the children of $\gamma$ be $\alpha$ and $\beta$.
\State Factorize
$
K_{\gamma} =
\begin{pmatrix}
V_{\alpha}^* Y_{\alpha} & I \\
I & V_{\beta}^* Y_{\beta}
\end{pmatrix}$
and \underline{store its LU factorization in-place}.
\State Apply $K_{\gamma}^{-1}$ (the previous LU factorization) to solve
\begin{equation} \label{e:node1}
\begin{pmatrix}
V_{\alpha}^* Y_{\alpha} & I \\
I & V_{\beta}^* Y_{\beta}
\end{pmatrix}
\begin{pmatrix}
W_{\alpha} \\
W_{\beta}
\end{pmatrix}
=
\begin{pmatrix}
V_{\alpha}^* Y_{\text{big}}(\mathcal{I}_{\alpha}, 1:r\ell) \\
V_{\beta}^* Y_{\text{big}}(\mathcal{I}_{\beta}, 1:r\ell)
\end{pmatrix}
.
\end{equation}
\State Compute
\begin{equation} \label{e:node2}
Y_{\text{big}}(\mathcal{I}_{\gamma}, 1:r\ell)
\gets
Y_{\text{big}}(\mathcal{I}_{\gamma}, 1:r\ell)
-
\begin{pmatrix}
Y_{\alpha} W_{\alpha} \\
Y_{\beta} W_{\beta}
\end{pmatrix}
.
\end{equation}
\EndFor
\EndFor
\end{algorithmic}
\end{algorithm}

\begin{algorithm}
\caption{Nonrecursive algorithm for solving \cref{e:hxb}: solution stage}
\label{a:solve}
\begin{algorithmic}[1]
\Require  matrix $Y_{\text{big}}$, matrix $V_{\text{big}}$, stored LU factorizations, cluster tree $\mathcal{T}_L$, and right-hand side $b$.
\Ensure solution of \cref{e:hxb}, namely, $x$.
%
\State $x \gets b$
\hfill \Comment \emph{// $x$ overwrites $b$.}
%
\For {leaf node $\alpha$ in $\mathcal{T}_L$}
\State Apply $D_{\alpha}^{-1}$ ({precomputed LU factorization}) to solve right-hand side $x(\mathcal{I}_{\alpha})$ {in-place}.
\EndFor
%
\For {level $\ell=L-1$ \textbf{to} $0$}
\For {tree node $\gamma$ at level $\ell$ in $\mathcal{T}_L$}
\State Let the children of $\gamma$ be $\alpha$ and $\beta$.
\State Apply $K_{\gamma}^{-1}$ ({precomputed LU factorization}) to solve
\begin{equation} \label{e:node3}
\begin{pmatrix}
V_{\alpha}^* Y_{\alpha} & I \\
I & V_{\beta}^* Y_{\beta}
\end{pmatrix}
\begin{pmatrix}
w_{\alpha} \\
w_{\beta}
\end{pmatrix}
=
\begin{pmatrix}
V_{\alpha}^* x(\mathcal{I}_{\alpha}) \\
V_{\beta}^* x(\mathcal{I}_{\beta})
\end{pmatrix}
.
\end{equation}
\State Compute
\begin{equation} \label{e:node4}
x(\mathcal{I}_{\gamma})
\gets
x(\mathcal{I}_{\gamma})
-
\begin{pmatrix}
Y_{\alpha} w_{\alpha} \\
Y_{\beta} w_{\beta}
\end{pmatrix}
.
\end{equation}
\EndFor
\EndFor
\end{algorithmic}
\end{algorithm}

\subsection{Parallel algorithms for GPUs} \label{s:gpu}

Recall that in the recursive algorithm, the two  subproblems in \cref{e:recursion} are independent of each other and can be solved in parallel. Consider the underlying cluster tree of a HODLR matrix. If we associate the root with the task of solving the original linear system and every other tree node with the task of solving a corresponding subproblem, then the cluster tree is effectively a \emph{task graph}. In the task graph, every node depends on its two children if they exist.
The same analysis applies to the factorization stage and the solution stage as described in \cref{a:fact,a:solve}, respectively, and the underlying task graphs are exactly the same (definitions of tasks differ).
%

Let us focus on the factorization stage, and most of the following discussion applies to the solution stage as well. It is clear that in the underlying task graph, all the nodes (tasks) at the same level are embarrassingly parallel. In other words, the two for-loops at lines 2\&7 in \cref{a:fact} can be parallelized. Indeed, the \hodlrlib{} library~\cite{ambikasaran2019hodlrlib} employs the ``parallel-for'' directive in OpenMP~\cite{chandra2001parallel} to parallelize the two for-loops. In \cref{a:fact}, it is also obvious that the nodes (tasks)  at the same level work on different subblocks in $U_{\text{big}}$ and $V_{\text{big}}$, which
 correspond to a (consecutive) partitioning of the row indices. 

Notice that in the factorization stage, every leaf task requires solving a linear system with multiple right-hand sides and  every nonleaf task comprises solving a linear system and matrix-matrix multiplications (gemm's). Therefore,  \cref{a:fact} can be transformed into a parallel algorithm by simply batching the BLAS and LAPACK calls from tasks at the same level. In particular, we take advantage of the batched LU factorization/solution (getrfBatched/getrsBatched) and the batched gemm (gemmBatched) from cuBLAS\footnote{\url{https://docs.nvidia.com/cuda/cublas/index.html}} in our GPU algorithms. With our new data structure that concatenates the left low-rank bases into one (big) matrix, we need only one cuBLAS call for computations involving multiple left low-rank bases across different levels in the tree. This property simplifies the implementation, avoids unnecessary data movement, and reduces the number of kernel launches. We present the pseudocode of our GPU factorization algorithm in \cref{a:gpuf}.

The new data structure also allows us to take advantage of an optimization of the general batched gemm kernel---gemmStridedBatched\footnote{\url{https://developer.nvidia.com/blog/cublas-strided-batched-matrix-multiply/}}---from cuBLAS when the (left) low-rank bases at the same level have the same sizes. In other words, there is a constant stride among the corresponding subblocks in $U_{\text{big}}$ and $V_{\text{big}}$ that are accessed by a batched gemm kernel. Such an optimization improves the performance significantly when the sizes of input matrices are small.

The solution stage as described in \cref{a:solve} has similar structure as the factorization stage. It is straightforward to see that the two for-loops at lines 2\&6 in \cref{a:solve} can be parallelized.
Again, we batch the BLAS and LAPACK calls from tasks at the same level. With the precomputed data from the factorization stage, the solution stage is relatively simple as described in \cref{a:gpus}.

\begin{table}
\caption{Notations for \cref{a:gpuf,a:gpus}}
\label{t:notation}
\begin{center}
\begin{tabular}{r c p{6cm} }
\toprule
superscript $\ell$ &  & block-row view of a matrix, where matrix rows are partitioned according to nodes at the $\ell$-th level in $\mathcal{T}_L$. \\ \midrule
operator $\odot$ &  & block-wise matrix multiplication between two block-row matrices (the result is a block-row matrix). \\ \midrule 
$Y^{\ell+1}$ & & shorthand for $Y_{\text{big}}^{\ell+1}(:,r\ell+1:r(\ell+1))$. \\
$V^{\ell+1}$ & & shorthand for $V_{\text{big}}^{\ell+1}(:,r\ell+1:r(\ell+1))$. \\
$\left(V^{\ell +1} \right)^*$ & & block-wise transpose of $V^{\ell +1}$. \\ \midrule
%
%
%
%
%
$\left( D_{\text{big}}^{L} \right)^{-1}$ & &   applying matrix inverse (solution of linear systems) with  block-wise LU factorizations.  \\
$\left( K_{\text{big}}^{\ell} \right)^{-1}$ & &   \\
%
\bottomrule
\end{tabular}
\end{center}
%
\end{table}

\begin{algorithm}
\caption{GPU algorithm for solving \cref{e:hxb}: factorization stage (see notations in \cref{t:notation})}
\label{a:gpuf}
\begin{algorithmic}[1]
\Require matrix $D_{\text{big}}$, matrix $U_{\text{big}}$, matrix $V_{\text{big}}$, and cluster tree $\mathcal{T}_L$.
\Ensure matrix $Y_{\text{big}}$, block-wise factorizations of matrices $D_{\text{big}}^L$, $K_{\text{big}}^0$, $K_{\text{big}}^1$, \ldots, $K_{\text{big}}^{L-1}$.
\Statex
%
%
\State $Y_{\text{big}} \gets  U_{\text{big}}$
\hfill \Comment \emph{// $Y_{\text{big}}$ overwrites $U_{\text{big}}$.}
\State \Call{{batched\_LU\_factorize\_inplace}}{$D_{\text{big}}^{L}$}
\State \textsc{{batched\_LU\_solve\_inplace:}}
\[
Y_{\text{big}}^{L} \gets \left( D_{\text{big}}^{L} \right)^{-1} \odot Y_{\text{big}}^{L}
\]
%
\For {level $\ell=L-1$ \textbf{to} $0$}
%
\State  {\textsc{batched\_gemm:}}
\[
T^{\ell +1} \gets \left(V^{\ell +1} \right)^* \odot Y^{\ell +1}
\]
\hfill \Comment \emph{// $T$ and $W$ (below) are temporary variables; see \cref{e:node1}.}
\State  {\textsc{batched\_gemm:}}
\[
W^{\ell +1} \gets \left(V^{\ell +1} \right)^* \odot Y_{\text{big}}^{\ell +1}(:,1:r\ell)
\]
\State Form matrix $K_{\text{big}}^{\ell}$ with $T^{\ell +1}$.
\hfill \Comment \emph{// See \cref{e:k} and \cref{f:dk}.}
\State \Call{{batched\_LU\_factorize\_inplace}}{$K_{\text{big}}^{\ell}$}
\State {\textsc{batched\_LU\_solve\_inplace:}}
\[
W^{\ell} \gets (K_{\text{big}}^{\ell})^{-1} \odot  W^{\ell}
\]
\State  {\textsc{batched\_gemm:}}
\[
Y_{\text{big}}^{\ell +1}(:,1:r\ell) \gets Y_{\text{big}}^{\ell +1}(:,1:r\ell) - Y^{\ell +1} \odot W^{\ell +1}
\]
\EndFor
\end{algorithmic}
\end{algorithm}

\begin{algorithm}
\caption{GPU algorithm for solving \cref{e:hxb}: solution stage (see notations in \cref{t:notation})}
\label{a:gpus}
\begin{algorithmic}[1]
\Require matrix $Y_{\text{big}}$, matrix $V_{\text{big}}$, (block-wise factorizations of) matrices $D_{\text{big}}^L$, $K_{\text{big}}^0$, $K_{\text{big}}^1$, \ldots, $K_{\text{big}}^{L-1}$, cluster tree $\mathcal{T}_L$, and right-hand side $b$.
\Ensure solution $x$.
%
\Statex
%
%
\State $x \gets b$
\hfill \Comment \emph{// $x$ overwrites $b$.}
\State \textsc{batched\_LU\_solve\_inplace:}
\[
x^{L} \gets \left( D_{\text{big}}^{L} \right)^{-1} \odot x^{L}
\]
\For {level $\ell=L-1$ \textbf{to} $0$}
%
%
 \State  \textsc{batched\_gemm:}
\[
w^{\ell +1} \gets \left(V^{\ell +1} \right)^* \odot x^{\ell +1}
\]
\hfill \Comment \emph{// $w$ is a temporary variable; see \cref{e:node3}.}
\State \textsc{batched\_LU\_solve\_inplace:}
\[
w^{\ell} \gets (K_{\text{big}}^{\ell})^{-1} \odot  w^{\ell}
\]
%
\State  \textsc{batched\_gemm:}
\[
x^{\ell +1} \gets x^{\ell +1} - Y^{\ell +1} \odot w^{\ell +1}
\]
\EndFor
\end{algorithmic}
\end{algorithm}

Notice that for the first few levels the number of nodes is small, and we empirically found that the batched gemm kernel was outperformed by launching independent gemm kernels using CUDA {streams}.
Before ending this section, we point out two variants of \cref{e:node}:
\begin{align*} \label{e:var}
&
\begin{pmatrix}
I & V_{\beta}^* Y_{\beta}  \\
V_{\alpha}^* Y_{\alpha} & I
\end{pmatrix}
\begin{pmatrix}
w_{\alpha} \\
w_{\beta}
\end{pmatrix}
=
\begin{pmatrix}
V_{\beta}^* z_{\beta} \\
V_{\alpha}^* z_{\alpha}
\end{pmatrix} \\
\text{and} \qquad
&
\begin{pmatrix}
I & V_{\alpha}^* Y_{\alpha} \\
V_{\beta}^* Y_{\beta}  & I
\end{pmatrix}
\begin{pmatrix}
w_{\beta} \\
w_{\alpha}
\end{pmatrix}
=
\begin{pmatrix}
V_{\alpha}^* z_{\alpha} \\
V_{\beta}^* z_{\beta}
\end{pmatrix}
,
\end{align*}
where the coefficient matrices have identities on the diagonals but either the right-hand side or the solution is reordered. Although \cref{e:node} is mathematically equivalent to the two alternatives, they may perform differently on a GPU. To be concrete, we typically need partial pivoting for the (batched) LU factorization (line 7 in \cref{a:gpuf}) with the formulation in \cref{e:node}, which impedes performance. The alternative formulations could circumvent this issue, but they require an extra cost of shuffling the right-hand side or the solution.


\subsection{Complexity analysis} \label{s:complexity}

Given a rank-$r$ HODLR matrix $A \in \mathbb{R}^{N \times N}$ with the underlying  cluster tree  $\mathcal{T}_L$, we assume the off-diagonal blocks in $A$ have the same rank $r$ and the diagonal blocks in $A$ have the same size $m = N / 2^L$. In practice, we usually prescribe $m$ to be a small constant independent of $N$, or equivalently we set $L = \bigO(\log(N))$. The following analysis and results can be easily extended to general cases where the ranks and sizes of diagonal blocks vary.

Let us first consider the storage of $A$ that includes
\begin{itemize}
\item
Diagonal blocks: $m^2 \times 2^L = m N$.

\item
Off-diagonal blocks ($U_{\text{big}}$ and $V_{\text{big}}$): $2 \times N \times  r \times L = 2 r N L $.
\end{itemize}
As explained earlier, the factorization of the HODLR matrix can be done in-place, which requires little extra memory. We summarize the above as the following theorem.

\begin{theorem}[Storage] \label{t:store}
The storage of the HODLR matrix $A$ and its factorization is
\[
m_f = mN + r N L = \bigO(r N \log(N)).
\]
\end{theorem}
Next, we consider the computational cost for factorizing matrix $A$ and summarize the results in a theorem.
\begin{itemize}
\item
Leaf level (lines 3\&4 in \cref{a:fact}):
\[
\frac{2}{3} m^3 \times 2^L + 2m^2 \times r \times L \times 2^L = \frac{2}{3} m^2 N + 2 m r N L,
\]
where factorizing an $m\times m$ matrix using the LU factorization requires $2/3m^3$ operations and solving one right-hand side using the LU factorization requires $2m^2$ operations with forward and backward substitutions.

\item
At level $\ell \, (1\le \ell \le L)$, we solve \cref{e:node1} and compute  \cref{e:node2}. Assuming the rank $r$ is small, the cost is dominated by conducting matrix-matrix multiplication\footnote{The computational cost for multiplying an $n\times k$ (real) matrix and a $k\times m$ (real) matrix is $2knm$ operations.} to form the right-hand sides in  \cref{e:node1} and the update in \cref{e:node2},
which is
\[
2  r^2 N \ell + 2 r^2 N \ell = 4 r^2 N \ell
\]
operations. Notice this cost is linear with respect to $\ell$.
\end{itemize}

\begin{theorem}[Factorization cost]
The factorization of the HODLR matrix $A$ requires the following amount of operations:
\begin{align*}
t_f & = \frac{2}{3} m^2 N + 2 m r N L + \sum_{\ell=1}^{L} 4 r^2 N \ell  \\
&=  \frac{2}{3} m^2 N + 2 m r N L + 2 r^2 N (L+L^2) \\
&= \bigO(r^2 N \log^2(N)).
\end{align*}
\end{theorem}
Finally, we consider the computational cost for solving an arbitrary right-hand side using the factorization of matrix $A$ and summarize the results in a theorem.
\begin{itemize}
\item
Leaf level (line 3 in \cref{a:solve}):
\[
 2m^2 \times 2^L = 2 m N,
\]
where $2^L$ is the number of leaf nodes in the cluster tree (diagonal blocks in $A$).

\item
At level $\ell \, (1\le \ell \le L)$, we solve \cref{e:node3} and compute  \cref{e:node4}. Assuming the rank $r$ is small, the cost is dominated by conducting {matrix-vector multiplication} to form the right-hand side in  \cref{e:node3} and the update in \cref{e:node4},
which is
\[
2  r N + 2 r N = 4 r N
\]
operations. Notice that this cost does not depend on $\ell$.
\end{itemize}

\begin{theorem}[Solution cost] \label{t:solve}
The solution of an arbitrary right-hand side using the factorization of $A$ requires the following amount of operations:
\begin{equation} \label{e:solution_scale}
t_s = 2 m N + \sum_{\ell=1}^{L}  4 r N  = 2 m N + 4 r N L = \bigO(r N \log(N)).
\end{equation}
\end{theorem}
It is not a coincidence that the amount of required operations in \cref{t:solve} is twice as many as the amount of storage in \cref{t:store}. The fact is that every stored entry is touched once and is involved with two operations (a multiplication and an addition) in the solution stage.

\begin{remark}[Numerical rank]

From users' perspective, it is usually more desirable to specify a tolerance for low-rank compressions rather than the ranks directly. Consider using HODLR approximations for problems that are not highly oscillatory. When the underlying problem lies in 1D, the ranks of all off-diagonal blocks are roughly the same
and are independent of the problem size for a prescribed tolerance. As the above analysis shows, the factorization and the solution both scale nearly linearly. However, when the underlying problem is in 2D or 3D, the ranks increase with the problem size, and the asymptotic complexities deteriorate (see, e.g., Remark 5.1 in~\cite{ambikasaran2014fast}). 

\end{remark}

\subsection{Two perspectives and connections} \label{s:connection}

The HODLR format is relatively simple among other formats that approximate off-diagonal blocks of a matrix, but the described algorithms have close connections to algorithms for more complicated formats. Below, let us  review two other perspectives on factorizing a HODLR matrix and discuss their connections to related algorithms.

\paragraph{Matrix factorization} The first perspective is from a matrix factorization point of view, which initially appeared in~\cite{ambikasaran2013mathcal}. Let us illustrate this with the following example. (An interesting extension is the computation of a symmetric factorization of a HODLR matrix that is SPD~\cite{ambikasaran2014fast}, which has various applications involving covariance matrices.)

\begin{example}
Consider the HODLR matrix $A$ in \cref{f:hodlr}, where the underlying cluster tree has three levels (level 0, 1, and 2). Suppose we apply \cref{a:fact} to $A$. Equivalently, we obtain the following factorization (and its ``inverse'')
\begin{align*}
A
& =
\begin{pmatrix}
A_{2} & U_{2} V_{3}^* \\
U_{3} V_{2}^* & A_{3}
\end{pmatrix}
 =
\begin{pmatrix}
A_{2} &  \\
 & A_{3}
\end{pmatrix}
\begin{pmatrix}
I & Y_{2} V_{3}^* \\
Y_{3} V_{2}^* & I
\end{pmatrix}
\\
& =
\underbrace{
\begin{pmatrix}
A_{4}   \\
 & A_{5} \\
  & & A_{6} \\
  & & & A_{7} \\
\end{pmatrix}
}_{
A^{(3)}
}
\underbrace{
\begin{pmatrix}
I & Y_4 V_5^* \\
Y_5 V_4^* & I \\
  & & I & Y_6 V_7^* \\
  & & Y_7 V_6^* & I \\
\end{pmatrix}
}_{
A^{(2)}
} \\
 & \quad \quad \quad \underbrace{
\begin{pmatrix}
I & Y_{2} V_{3}^* \\
Y_{3} V_{2}^* & I
\end{pmatrix}
}_{
A^{(1)}
}
.
\end{align*}
Notice that the $Y$ matrices are among the outputs of \cref{a:fact}. In addition, we can apply $A^{-1}$ easily because we can apply the inverses of $A^{(1)}$, $A^{(2)}$, and $A^{(3)}$ easily. For example, we have stored the factorizations of every diagonal block in $A^{(3)}$ (line 3 in  \cref{a:fact}), and every $2\times 2$ diagonal block in $A^{(2)}$ can be inverted according to \cref{e:wood}, where the factorization of a small matrix has been stored (line 9 in  \cref{a:fact}). It should not be surprising that the algorithm of applying $A^{-1}$ to a vector $b$ is just \cref{a:solve}. We summarize the discussion into the following:

\begin{theorem}
\cref{a:fact} and \cref{a:solve} are equivalent to computing a matrix factorization of a HODLR matrix and applying the inverse of the factorization to a vector, respectively.
\end{theorem}

Notice that the above matrix factorization of $A$ naturally leads to an efficient algorithm for evaluating the determinant of $A$. Observe that
\[
\text{det}(A) = \text{det}(A^{(1)}) \text{det}(A^{(2)}) \text{det}(A^{(3)}),
\]
where $\text{det}(A^{(3)})$ can be readily obtained with the LU factorizations of its diagonal blocks, the determinant of every $2\times 2$ diagonal block in $A^{(2)}$ and $A^{(1)}$ can be evaluated efficiently with the stored factorizations according to Sylvester’s determinant theorem~\cite{ambikasaran2015fast}.
\end{example}
%
%
%
%

\paragraph{Extended sparse linear system} The other perspective is related to sparse matrix algebra, where we embed a HODLR matrix into a larger sparse matrix~\cite{aminfar2016fast}. To be more specific, the original dense problem \cref{e:hxb} is equivalent to a larger sparse linear system, and we solve the sparse problem with straightforward Gaussian elimination to obtain the solution of the original problem. The power of this approach is that it also generalizes to solving problems involving more complicated formats of hierarchical low-rank compression. Let us illustrate the main idea through the following example.

\begin{example}
Consider solving a dense linear system
\[
\begin{pmatrix}
A_{2} & U_{2} V_{3}^* \\
U_{3} V_{2}^* & A_{3}
\end{pmatrix}
\begin{pmatrix}
x_{2} \\
x_{3}
\end{pmatrix}
=
\begin{pmatrix}
b_{2} \\
b_{3}
\end{pmatrix}
.
\]
With two auxiliary variables $w_{3} = V_{2}^* x_{2}$ and $w_{2} = V_{3}^* x_{3}$ (the subscripts of $w$'s are chosen for the convenience of the following presentation), we rewrite the above dense linear system as the following block-sparse linear system:
\[
\begin{pmatrix}
A_{2} & & U_{2}  \\
& A_{3} & & U_{3}  \\
V_{2}^* & & & - I \\
& V_{3}^* &  - I
\end{pmatrix}
\begin{pmatrix}
x_{2} \\
x_{3} \\
w_{2} \\
w_{3}
\end{pmatrix}
=
\begin{pmatrix}
b_{2} \\
b_{3} \\
0 \\
0
\end{pmatrix}
.
\]
The beauty is that we can solve this sparse linear system with Gaussian elimination in a straightforward manner: we factorize the block-sparse matrix with the block-LU factorization and then solve with block-forward and block-backward substitutions. To be concrete, we first apply two steps of block Gaussian elimination (for $x_2$ and $x_3$) and obtain the resulting linear system involving the resulting Schur complement
\[
\begin{pmatrix}
V_{2}^* Y_{2} &  I \\
I  & V_{3}^* Y_{3}
\end{pmatrix}
\begin{pmatrix}
w_{2} \\
w_{3}
\end{pmatrix}
=
\begin{pmatrix}
V_{2}^* z_{2} \\
V_{3}^* z_{3}
\end{pmatrix}
,
\]
where $Y_{2} = (A_{2})^{-1} U_{2}$, $Y_{3} = (A_{3})^{-1} U_{3}$, $z_{2} = (A_{2})^{-1} b_{2}$, and $z_{3} = (A_{3})^{-1} b_{3}$. Readers may recognize  this small linear system as \cref{e:node}. After solving the above system, we get the original solution via forward and backward substitutions, which  turns out to be exactly the same as \cref{e:solution}:
\[
\begin{pmatrix}
x_{2} \\
x_{3}
\end{pmatrix}
=
\begin{pmatrix}
z_{2} \\
z_{3}
\end{pmatrix}
-
\begin{pmatrix}
Y_{2} w_{2} \\
Y_{3} w_{3}
\end{pmatrix}
.
\]
\end{example}

This approach extends to solving linear systems involving HSS matrices\footnote{If the low-rank bases in a HODLR matrix are \emph{nested} (the bases of a nonleaf node can be constructed from those of its two children), then the HODLR matrix is an HSS matrix.}. It is shown in~\cite{ambikasaran2013fast} (Section 7.2.3) that applying Gaussian elimination to an extended sparse linear system introduces \emph{no} fill-in with an appropriate ordering. Alternatively, we can feed an extended sparse linear system directly to a highly optimized sparse direct solver as proposed by Ho and Greengard~\cite{ho2012fast} for computational efficiency and stability.

Both the HODLR and the HSS formats rely on the so-called \emph{weak admissibility}, and related algorithms are relatively simple. By contrast, the compression format associated with the FMM employs the \emph{strong admissibility}, and solving a linear system involving such a matrix is still an active research question. Recent advances~\cite{ambikasaran2014inverse,coulier2017inverse,takahashi2020parallelization} demonstrate that the approach of creating an extended sparse linear system is still viable as long as we compress fill-in blocks during Gaussian elimination. It turns out that this improvement can also be applied to solve large \emph{sparse} linear systems arising from the discretization of linear elliptic partial differential equations that are not highly indefinite~\cite{pouransari2017fast,sushnikova2018compress,chen2018distributed,chen2019robust}.

\section{Numerical Results}

We benchmark the speed and the accuracy of our GPU solver and compare it against existing methods in the literature.
In \cref{s:kernel}, we apply our GPU solver to kernel matrices and compare to \hodlrlib~\cite{ambikasaran2019hodlrlib}.
In \cref{s:laplace,s:helmholtz}, we apply our HODLR solver to linear systems coming from the discretization of  BIEs with a Laplace double-layer potential and a Helmholtz potential, respectively. We compare to the block-sparse solver proposed by Ho and Greengard~\cite{ho2012fast}, which can leverage state-of-the-art sparse direct solvers.
%
The machine where our numerical experiments were performed has
\begin{itemize}

\item
two Intel Xeon Gold 6254 CPUs, each with 18 cores at 3.10 GHz (peak performance $\approx$ 1.27 TFlop/s),

\item
an NVIDIA GPU that is a Tesla V100 GPU with 32 GB of memory (peak performance $\approx$ 7 TFlop/s).

\item
{a PCIe 3.0 $\times 16$ between the CPU and the GPU that can deliver up to 15.75 GB/s.}

\end{itemize}
Our GPU code was compiled with the NVIDIA compiler {nvcc} (version 11.3.58) and linked with the cuBLAS library (version 11.4.2.10064) on the Linux OS (5.4.0-72-generic.x86\_64). Calculations are performed with double-precision floating-point arithmetic unless stated otherwise. Notations that we use to report results are shown in \cref{t:notation2}, and the timings were taken as the average of five consecutive runs.
%

\begin{table}
\caption{Notations for numerical results}
\label{t:notation2}
\begin{center}
\begin{tabular}{r c p{7cm} }
\toprule
$N$ && problem size/matrix size/degrees of freedom/number of points \\
$t_f$ && factorization time in seconds \\
$t_s$ && solution time for a random right-hand side in seconds \\
mem && memory of the factorization in gigabytes (GB) \\
relres && relative residual of a solution $\tilde{x}$, i.e., $ \|b - A \tilde{x} \| / \|b\|$ \\
\bottomrule
\end{tabular}
\end{center}
\end{table}

\subsection{{Kernel matrix}} \label{s:kernel}

Consider solving \cref{e:axb}, where matrix $A$ is generated
%
%
from the Rotne-Prager-Yamakawa (RPY) tensor kernel~\cite{rotne1969variational,yamakawa1970transport}, which is frequently used to model the hydrodynamic interactions in simulations of Brownian dynamics. Given a set of points $\{{y}_i\}_{i=1}^N$, the  kernel matrix $A$ is defined as
\begin{equation} \label{e:rpy}
A_{i,j} =
\left\{
\begin{array}{l c }
\frac{kT}{8 \pi \eta |r|} \left[ I + \frac{r \bigotimes r}{|r|^2}  + \frac{2a^2}{3|r|^2} ( I - 3\frac{r \bigotimes r}{|r|^2} ) \right] & \text{ if } r \ge 2a, \\
\frac{kT}{6 \pi \eta a} \left[ (1 - \frac{9}{32} \frac{|r|}{a}) I + \frac{3}{32a} \frac{r \bigotimes r}{|r|} \right] & \text{ if } r < 2 a,
\end{array}
\right.
\end{equation}
where $r = y_i - y_j$ and $|r|$ denotes its two norm.

We approximated matrix $A$ with a HODLR matrix and obtained an approximate solution for \cref{e:axb}. In the following, we compare our GPU solver to the \hodlrlib{} library~\cite{ambikasaran2019hodlrlib}, which is a  C++ code parallelized with OpenMP for shared-memory multicore architectures. To be consistent with the benchmark of \hodlrlib{}\footnote{\url{https://hodlrlib.readthedocs.io/en/latest/benchmarks.html}}, we randomly generated $\{{x}_i\}_{i=1}^N$ with a uniform distribution over $[-1,1]$, and set $k = T= \eta = 1$, $a = |r|_{\min}/2$ with $|r|_{\min}$ being the minimum distance among all pairs of $x_i$ and $x_j$. We compiled \hodlrlib{} with g++ 7.5.0 and linked it with the (sequential) Intel MKL library of version 20.0. Timing results of \hodlrlib{} were obtained on two Intel Xeon CPUs of 36 cores combined.


{We constructed HODLR approximations on the CPUs, where diagonal blocks have size $64 \times 64$ and off-diagonal blocks are compressed using the LowRank::rookPiv() function, an approximate partial-pivoted LU, in \hodlrlib{}.} For our GPU solver, we  copied data including $D_{\text{big}}$, $U_{\text{big}}$, and $V_{\text{big}}$ to the GPU, {where we typically used a significant portion of the bandwidth, around 12 GB/s, in our experiments.}

\begin{table}
    \caption{\em Results for solving \cref{e:axb} with the RPY kernel defined in \cref{e:rpy}. The low-rank approximation accuracy is prescribed to be $10^{-12}$.}
    \label{t:hodlrlib}
    \centering
    \begin{tabular}{c|cc|cc|cc} \toprule
	\multirow{2}{*}{$N$} & \multicolumn{2}{c|}{\hodlrlib{}} & \multicolumn{2}{c|}{GPU Solver} & \multirow{2}{*}{mem} & \multirow{2}{*}{relres}  \\
	& $t_f$ & $t_s$ & $t_f$ & $t_s$ & & \\ \midrule
	\rowcolor{Gray}
	$2^{17} = 131,072$ & 1.47 & 0.22 & 7.39e-2 & 4.37e-3 & 0.88 & 1.68e-11  \\
	$2^{18} = 262,144$ & 5.09 & 0.61 & 1.81e-1 & 7.43e-3 & 1.93 & 2.57e-9  \\
	\rowcolor{Gray}
	$2^{19} = 524,288$ & 10.9 & 1.26 & 3.86e-1 & 1.27e-2 & 4.23 & 5.28e-11 \\
	$2^{20} = 1,048,576$ & 23.1 & 2.76 & 7.75e-1 & 2.12e-2 & 8.94 & 1.32e-9  \\
	\rowcolor{Gray}
	$2^{21} = 2,097,152$ & 51.7 & 5.42 & 1.89e-0 & 4.23e-2 & 19.2 & 1.10e-9  \\
	\bottomrule
    \end{tabular}
\end{table}

\begin{figure}
\centering
\begin{subfigure}{0.24\textwidth}
\centering
\includegraphics[width=\textwidth]{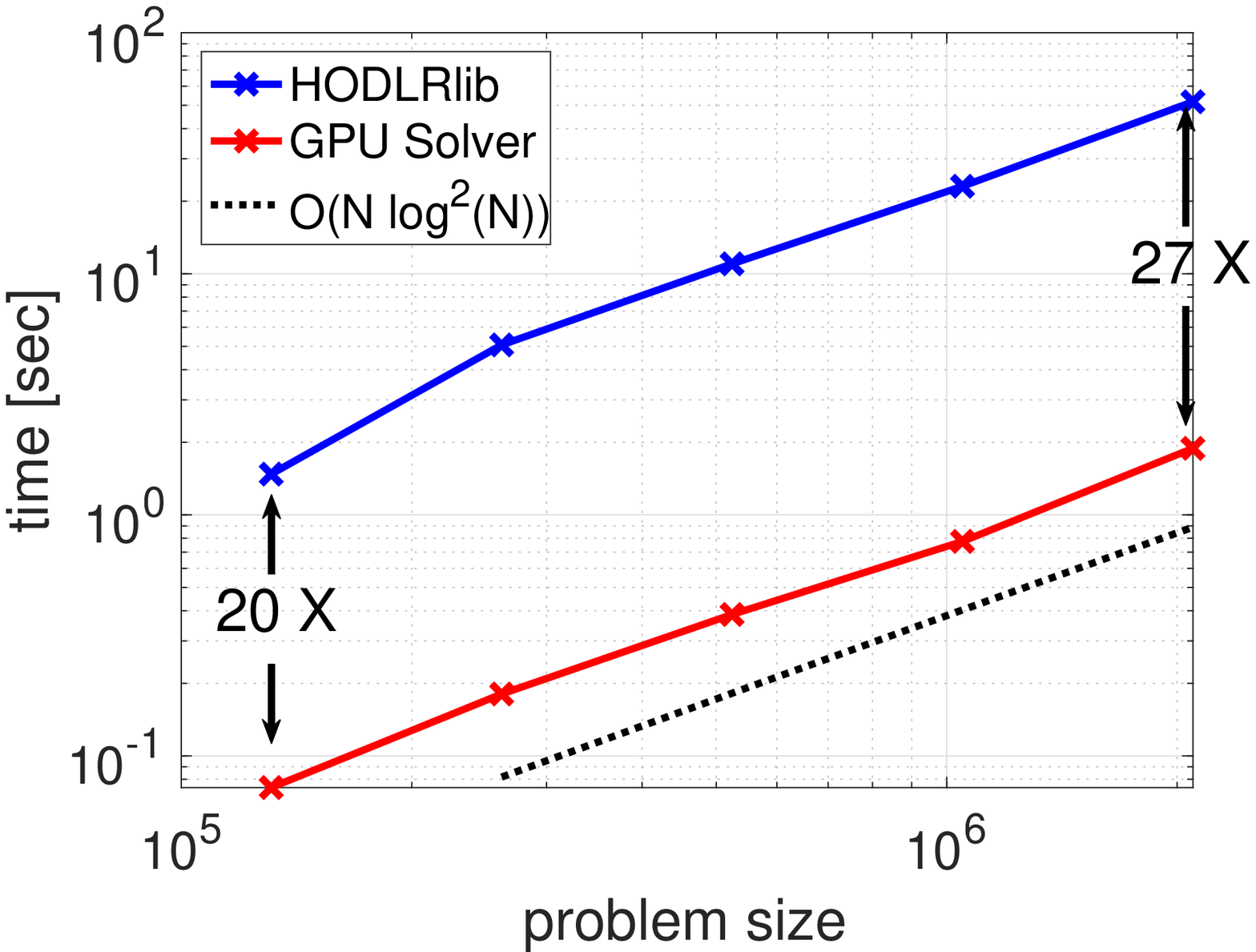}
\caption{Factorization time}
\end{subfigure}
\begin{subfigure}{0.24\textwidth}
\centering
\includegraphics[width=\textwidth]{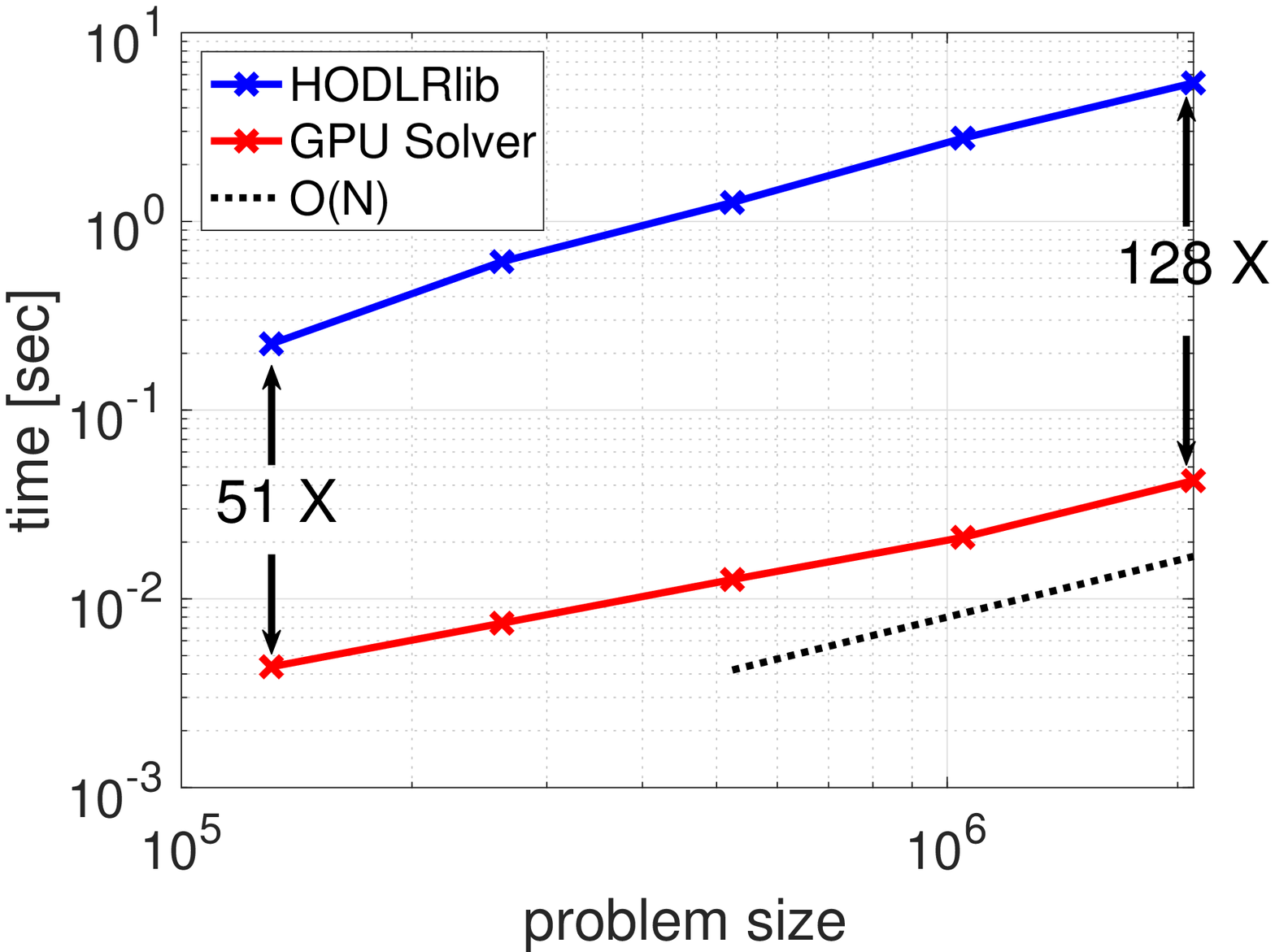}
\caption{Solution time}
\end{subfigure}
\caption{Comparison between (1) the \hodlrlib{} library~\cite{ambikasaran2019hodlrlib} running on two Intel Xeon Gold 6254 18-core CPUs and (2) our GPU solver running on an NVIDIA Tesla V100 GPU.}
\label{f:holderlib}
\end{figure}

With a prescribed accuracy of $10^{-12}$ for low-rank compression, we obtained the results in \cref{t:hodlrlib}, where the factorization time and the solution time are plotted in  \cref{f:holderlib}. From these results, we make two observations. First, the factorization time, the solution time, and the memory footprint all scale nearly linearly as the theory in \cref{s:complexity} predict. The increase of the solution time is slower than $N\log(N)$ as given in \cref{e:solution_scale} and is close to being linear. The reason is that the GPU utilization increases as the problem size $N$ increases.
 Second, our GPU solver achieved significant speedups against \hodlrlib{}. In particular, the speedup becomes larger as the problem size increases, and the acceleration of the solution time is larger than that of the factorization time. As mentioned earlier, the parallelization in \hodlrlib{} is only across nodes at the same level (no parallelization within a tree node), while our GPU solver also employs parallelization inside every tree node. For $N=2^{21}$, the factorization and the solution of our GPU solver achieved 878 Gflop/s and 119 Gflop/s, respectively.




\subsection{Laplace equation on an exterior domain} \label{s:laplace}

\begin{table*}[t]
    \caption{\em Results for solving \cref{e:laplace_int} discretized with a 2nd-order quadrature. In (b), all calculations were performed in single precision except for the sequential block-sparse solver (Matlab backslash for sparse matrices does not support single precision).}
    \label{t:bie_dr}
    \begin{subtable}[h]{\textwidth}
    \centering
    \begin{tabular}{c|ccc|ccc|ccc|ccc|c} \toprule
	\multirow{2}{*}{$N$} & \multicolumn{3}{c|}{Serial HODLR Solver} & \multicolumn{3}{c|}{Serial Block-Sparse Solver} & \multicolumn{3}{c|}{Parallel Block-Sparse Solver} & \multicolumn{3}{c|}{GPU HODLR Solver} & \multirow{2}{*}{relres} \\
	& $t_{f}$ & $t_s$ & mem & $t_{f}$ & $t_s$ & mem & $t_{f}$ & $t_s$ & mem & $t_f$ & $t_s$ & mem & \\ \midrule
	\rowcolor{Gray}
	$2^{18} = 262,144$ & 4.51e+1 & 5.93e-1 & 1.09 & 2.87e+0 & 1.33e-1 & 0.57 & 7.03e+0 & 1.85e-2 & 3.56 & 6.94e-2 & 4.87e-3 & 1.09 & 2.10e-9 \\
	$2^{19} = 524,288$ & 9.73e+1 & 1.05e-0 & 2.25 & 5.88e+0 & 2.86e-1 & 1.14 & 1.37e+1 & 3.74e-2 & 7.08 & 1.40e-1 & 8.19e-3 & 2.25 & 7.13e-9 \\
	\rowcolor{Gray}
	$2^{20} = 1,048,576$ & 2.20e+2 & 2.18e-0 & 4.63 & 1.21e+1 & 5.09e-1 & 2.28 & 2.89e+1 & 8.30e-2 & 14.2 & 2.90e-1 & 1.28e-2 & 4.63 & 5.60e-9 \\
	$2^{21} = 2,097,152$ & 4.76e+2 & 4.99e-0 & 9.46 & 2.35e+1 & 1.00e-0 & 4.56 & 6.20e+1 & 1.82e-1 & 28.6 & 6.10e-1 & 2.40e-2 & 9.46 & 7.82e-9 \\
	\rowcolor{Gray}
	$2^{22} = 4,194,304$ & 1.05e+2 & 9.81e-0 & 19.3 & 4.90e+1 & 2.29e-0 & 9.15 & 1.29e+2 & 5.18e-1 & 56.9 & 1.25e-0 & 4.61e-2 & 19.3 & 1.31e-8 \\
    \bottomrule
    \end{tabular}
    \caption{\em High-accuracy solver}
    \label{t:bie_dr_acc10}
    \end{subtable}
    \begin{subtable}[h]{\textwidth}
    \centering
    \begin{tabular}{c|ccc|ccc|ccc|ccc|cc} \toprule
	\multirow{2}{*}{$N$} & \multicolumn{3}{c|}{Serial HODLR Solver} & \multicolumn{3}{c|}{Serial Block-Sparse Solver} & \multicolumn{3}{c|}{Parallel Block-Sparse Solver} & \multicolumn{3}{c|}{GPU HODLR Solver} & \multirow{2}{*}{relres} \\
	& $t_f$ & $t_s$ & mem & $t_f$ & $t_s$ & mem & $t_f$ & $t_s$ & mem & $t_f$ & $t_s$ & mem & \\ \midrule
	\rowcolor{Gray}
	$2^{18} = 262,144$ & 1.41e+0 & 3.07e-1 & 0.27 & 1.66e+0 & 8.44e-2 & 0.35 & 4.50e+1 & 9.80e-3 & 2.40 & 1.74e-2 & 2.66e-3 & 0.27 & 3.13e-5 \\
	$2^{19} = 524,288$ & 2.95e+0 & 5.82e-1 & 0.55 & 3.32e+0 & 1.68e-1 & 0.69 & 9.42e+1 & 1.77e-2 & 4.79 & 3.39e-2 & 3.92e-3 & 0.55 & 1.49e-4 \\
	\rowcolor{Gray}
	$2^{20} = 1,048,576$ & 5.88e+0 & 1.16e-0 & 1.09 & 6.55e+0 & 3.42e-1 & 1.39 & 1.99e+1 & 3.71e-2 & 9.56 & 5.79e-2 & 6.48e-3 & 1.09 & 7.20e-5 \\
	$2^{21} = 2,097,152$ & 1.21e+1 & 2.48e-0 & 2.13 & 1.32e+1 & 6.89e-1 & 2.77 & 3.86e+1 & 7.83e-2 & 19.1 & 1.29e-1 & 1.09e-2 & 2.13 & 6.11e-4 \\
	\rowcolor{Gray}
	$2^{22} = 4,194,304$ & 2.47e+1 & 5.40e-0 & 4.26 & 2.79e+1 & 1.36e-0 & 5.55 & 8.49e+1 & 2.04e-1 & 38.3 & 2.70e-1 & 2.05e-2 & 4.26 & 2.07e-4 \\
	$2^{23} = 8,388,608$ & 5.03e+1 & 1.08e+1 & 8.45 & 5.81e+1 & 2.87e-0 & 11.1 & 1.73e+2 & 3.99e-1 & 76.2 & 4.26e-1 & 4.06e-2 & 8.45 & 4.04e-4 \\
	\rowcolor{Gray}
	$2^{24} = 16,777,216$ & 1.19e+2 & 1.93e+1 & 17.0 & 1.18e+2 & 6.23e-0 & 23.1 & 3.77e+2 & 7.83e-1 & 152 & 8.58e-1 & 8.38e-2 & 17.0 & 7.12e-4 \\
    \bottomrule
    \end{tabular}
    \caption{\em Low-accuracy solver}
    \label{t:bie_dr_acc6}
    \end{subtable}
\end{table*}

\begin{figure}
\centering
\includegraphics[width=0.24\textwidth]{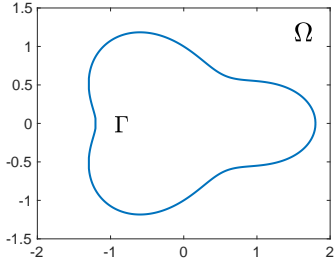}
\caption{A smooth contour $\Gamma$ on the {plane}. The problem domain $\Omega$,  where PDEs \cref{e:laplace,e:helmholtz} are defined, is \emph{exterior} to $\Gamma$.}
\label{f:contour}
\end{figure}

We next consider a boundary integral equation obtained by rewriting the BVP
\begin{equation} \label{e:laplace}
\left\{
\begin{array}{cl}
- \Delta u(x) = 0 & \text{ in } \Omega, \\
u(x) = f(x) & \text{ on } \Gamma,
\end{array}
\right.
\end{equation}
where $\Omega$ is the infinite domain that is exterior to the smooth contour $\Gamma$ shown in \cref{f:contour}.
To ensure a unique solution that is physically meaningful, we also require that there exists a real number $Q \in \mathbb{R}$ such that
\begin{equation} \label{e:laplace_inf}
\lim_{|x| \to \infty} \left( u(x) + \frac{Q}{2 \pi} \log |x| \right) = 0.
\end{equation}
We reformulate \cref{e:laplace} as
\begin{equation} \label{e:laplace_int}
\frac{1}{2} \sigma(x) + \int_{\Gamma} \left( d(x,y) - \frac{1}{2 \pi} \log|x-z| \right) \sigma(x) d s(y) = f(x),
\end{equation}
where 
$$
d(x,y) = \frac{n(y) \cdot (x-y)}{2\pi|x-y|^{2}}
$$
is the ``double layer kernel'' associated with the Laplace equation, where $n(y)$ is a normal vector at a point $y \in \Gamma$, and
where $z$ is a fixed point in the interior of $\Gamma$~\cite{kress1989linear,mclean2000strongly}. 
We chose $z$ to be the origin and discretized \cref{e:laplace} using the Trapezoidal rule.
To solve the resulting linear system, {we constructed a HODLR approximation using the proxy surface technique (see, e.g., \cite[Chapter~17]{martinsson2019fast}) on a CPU and copied the HODLR matrix to the GPU.}

For comparison, we implemented the block-sparse solver by Ho and Greengard~\cite{ho2012fast}. {The block-sparse solver builds  an HSS approximation~\cite{gillman2012direct} in a block sparse format and employs a sparse direct solver to solve the block sparse matrix.} This approach takes advantage of the stability and the efficiency of state-of-the-art sparse direct solvers. We refer interested readers to~\cite{ho2012fast} for comparisons between the block-sparse solver and an iterative solver employing the fast multipole method. For our problems, we empirically found that the Matlab backslash (calling UMFPACK~\cite{davis1997unsymmetric,davis2004algorithm}) and the MKL PARDISO solver~\cite{schenk2000efficient,schenk2004solving} performed best on a single core and on all 36 cores of two Intel Xeon CPUs, respectively\footnote{For UMFPACK, we used the [L,U,p] = lu(A,`vector') routine in Matlab, which applied only row-permutation on the input sparse matrix and avoided the overhead of computing a column permutation. {For parallel performance on the two Intel Xeon Gold 6254 CPUs, we compared several popular sparse direct solvers for our matrices}, and our conclusion is consistent with observations in the literature (see, e.g.,~\cite{kwack2016performance}).}. We refer to them as the sequential and the parallel block-sparse solvers.

\cref{t:bie_dr} shows the comparison between the block-sparse solver and our GPU solver, where the factorization time and the solution time are plotted in \cref{f:bie_dr}. {In practice, it is not necessary to use as many discretization points for \cref{e:laplace_int} on the smooth contour $\Gamma$. Here, our purpose is to benchmark the speed and the accuracy of our GPU solver.} From the results, we make the following observations. First, the scaling of the factorization, the solution time, and the memory footprint of our GPU solver increase nearly linearly with the problem size.

Second, our GPU solver achieved significant speedups compared to the block-sparse solver, as shown in \cref{f:bie_dr}. Although the sequential HODLR solver is slower than the sequential block-sparse solver, it is well understood that parallelizing a sparse direct solver, which the block-sparse solver relies on, is challenging~\cite{swirydowicz2021linear}. In fact, the parallel factorization stage, which consists of a symbolic factorization and a numerical factorization, is slower than the sequential counterpart. What happened here was that we did not compute any fill-in reducing ordering in the sequential method (the block-sparse matrix has a special structure that natural ordering turns out working well), but the parallel method spent a lot of time (and storage) in symbolic factorization to find parallelism (so the subsequent numerical factorization and the solution can be accelerated). Compared to the sequential method, the parallel block-sparse solver also consumed more memory from the analysis in the symbolic factorization phase that the solver needs in the numerical factorization and solve phases.

Third, one advantage of our method is that the accuracy is tunable. With highly accurate low-rank compressions, we obtained fast direct solvers as shown in \cref{t:bie_dr_acc10}. With reasonably accurate low-rank compressions, we obtained results in \cref{t:bie_dr_acc6}, where we were able to take advantage of single-precision floating-point arithmetic. {The use of single-precision calculations accelerated the factorization and the solution by approximately $2\times$, and it also reduced the memory footprint by $2\times$.}


\begin{figure}
\centering
\begin{subfigure}{0.24\textwidth}
\centering
\includegraphics[width=\textwidth]{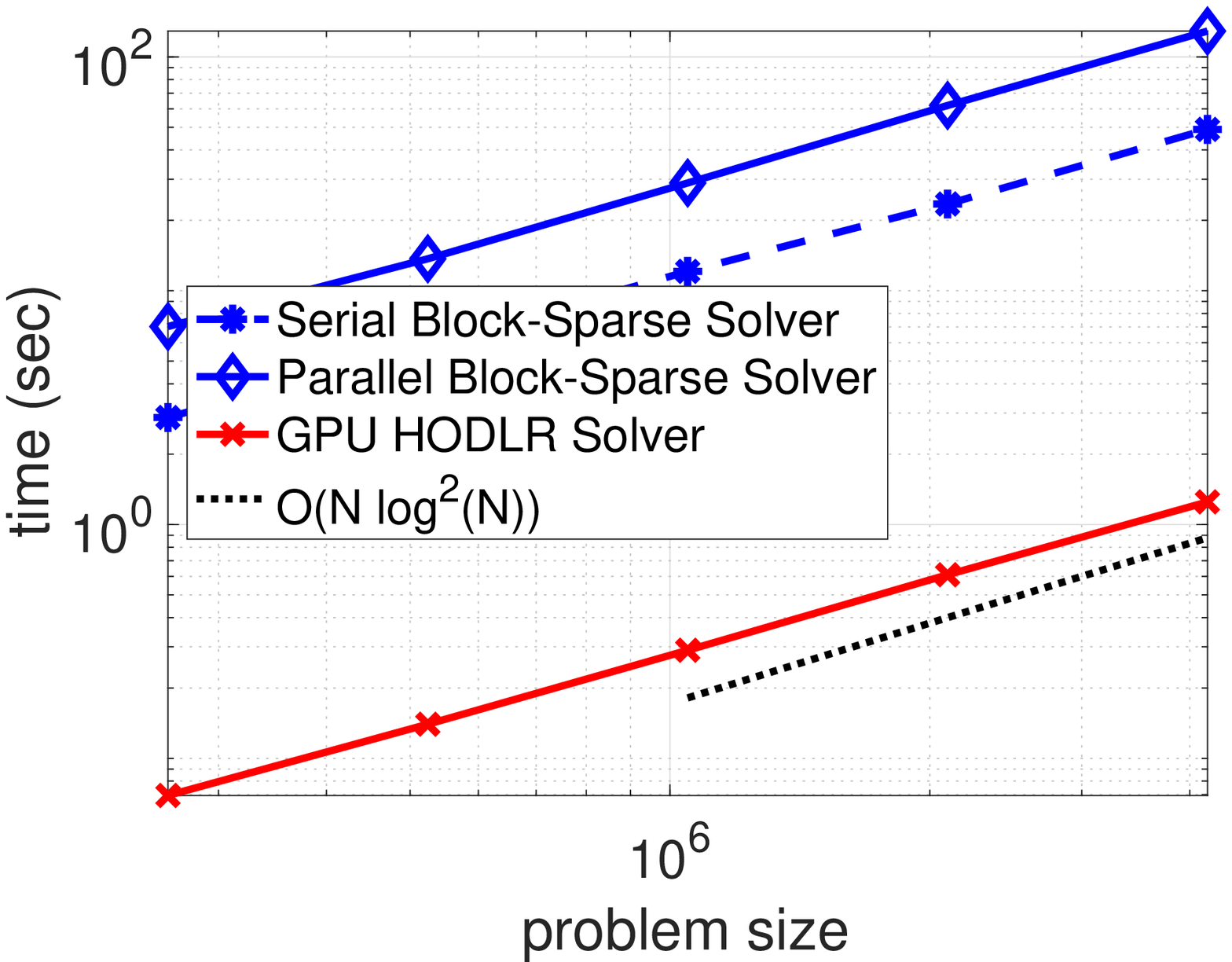}
\caption{High-accuracy factorization}
\end{subfigure}
\begin{subfigure}{0.24\textwidth}
\centering
\includegraphics[width=\textwidth]{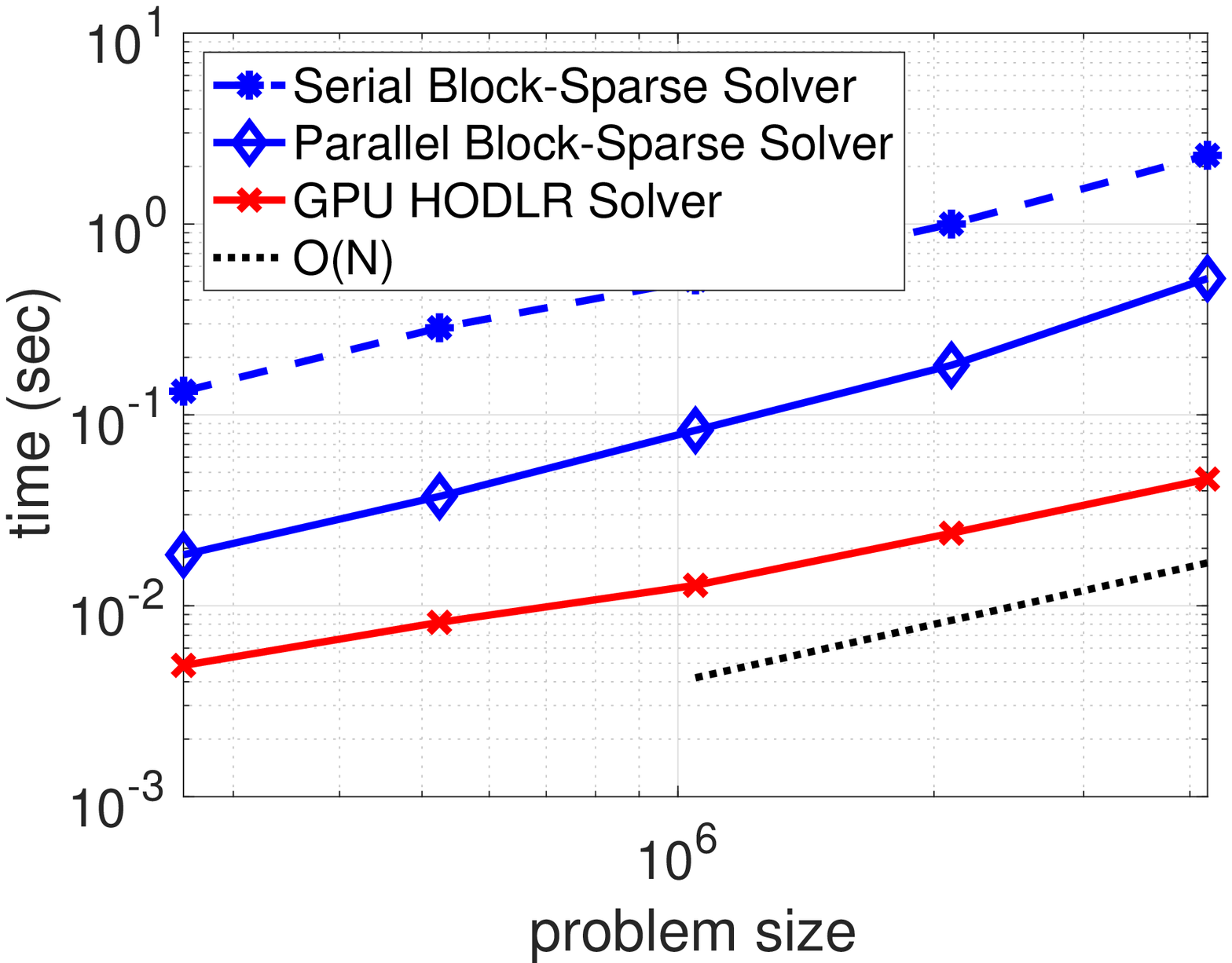}
\caption{High-accuracy solution}
\end{subfigure}
%
%
%
\begin{subfigure}{0.24\textwidth}
\centering
\includegraphics[width=\textwidth]{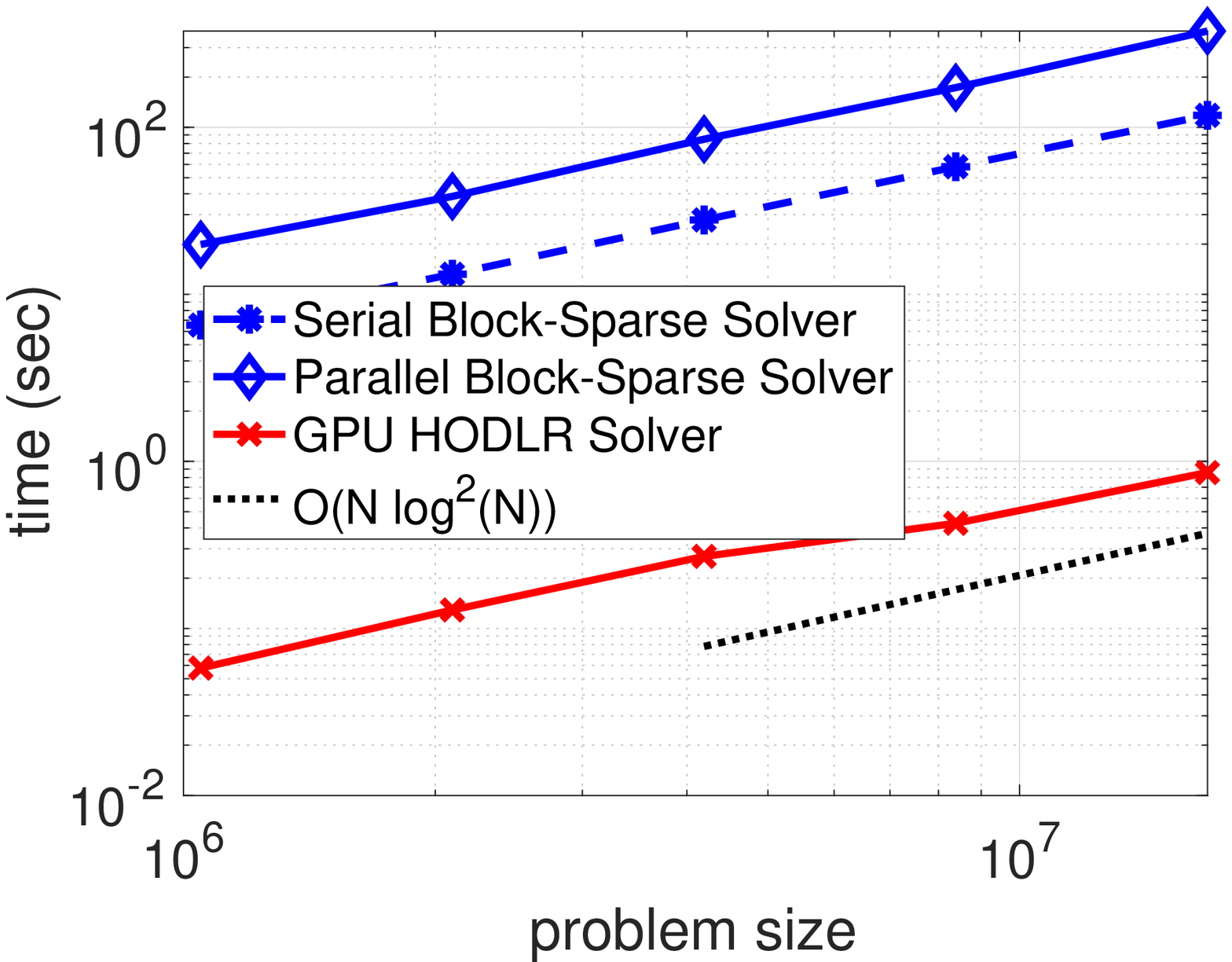}
\caption{Low-accuracy factorization}
\end{subfigure}
\begin{subfigure}{0.24\textwidth}
\centering
\includegraphics[width=\textwidth]{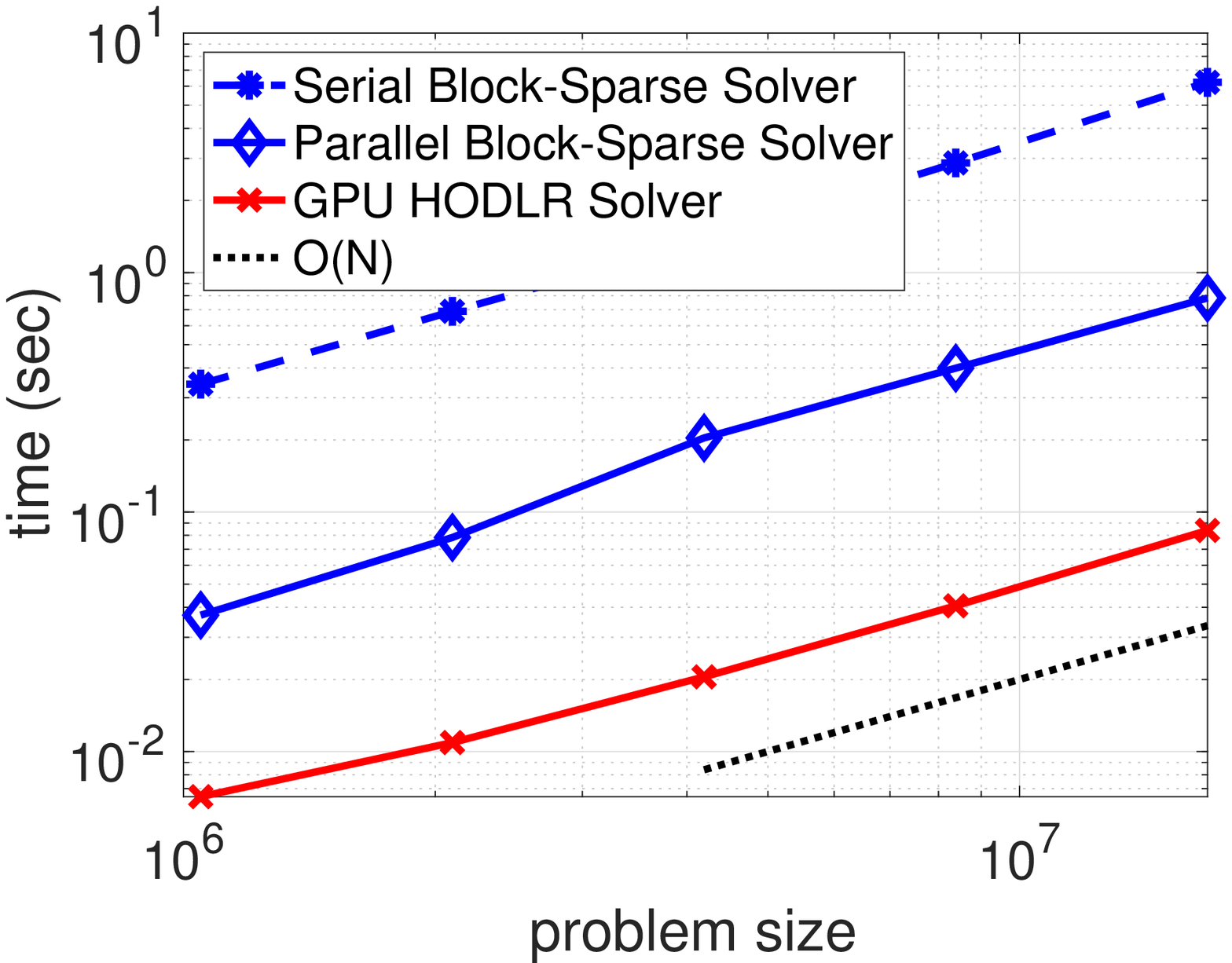}
\caption{Low-accuracy solution }
\end{subfigure}
\caption{Comparison between the block-sparse solver~\cite{ho2012fast} and our GPU solver for solving \cref{e:laplace_int} discretized with a 2nd-order quadrature.}
\label{f:bie_dr}
\end{figure}

\subsection{Helmholtz equation on an exterior domain} \label{s:helmholtz}

Consider the following {exterior} BVP that models, e.g., time-harmonic wave problems,
\begin{equation} \label{e:helmholtz}
\left\{
\begin{array}{cl}
- \Delta u(x) - \kappa^2 u(x) = 0 & \text{ in } \Omega, \\
u(x) = f(x) & \text{ on } \Gamma.
\end{array}
\right.
\end{equation}
Observe that \cref{e:helmholtz} is defined on an 
\emph{infinite} domain, as shown in \cref{f:contour}. To ensure a well-posed problem, we also require the radiation condition at infinity: for every unit vector $z$
\begin{equation} \label{e:helmholtz_inf}
\lim_{r \to \infty} \sqrt{r} \left( \frac{\partial u(rz)}{\partial r} - i \kappa u(rz) \right) = 0.
\end{equation}
Solving \cref{e:helmholtz}  by directly discretizing the PDE is quite challenging. Instead, we use a BIE formulation.
Define  the single- and double-layer kernels
\begin{align*}
s_{\kappa}(x,y) &= \phi_{\kappa}(x - y), \\
d_{\kappa}(x,y) &= n(y) \cdot \nabla_y \phi_{\kappa}(x - y),
\end{align*}
where $\phi_{\kappa}(x) = \frac{i}{4} H_0^{(1)}(\kappa |x|)$ is the fundamental solution of the Helmholtz operator and $H_0^{(1)}$ is the zeroth-order Hankel function of the first kind, and where $n(y)$ is the inward normal at $y\in\Gamma$. We reformulate \cref{e:helmholtz} as a second-kind Fredholm equation~\cite{kress1989linear,mclean2000strongly}
\begin{equation} \label{e:helmholtz_int}
\frac{1}{2} \sigma(x) + \int_{\Gamma} \left( d_{\kappa}(x,y) + i \eta s_{\kappa}(x,y) \right) \sigma(x) d s(y) = f(x),
\end{equation}
where $x \in \Gamma$ and $\eta$ is a parameter that is often chosen as $\eta = \pm \kappa$. The BIE is defined on a finite domain and automatically satisfies \cref{e:laplace_inf}. We chose {$\eta = \kappa = 100$} and discretized \cref{e:helmholtz_int} with the 6-th order quadrature proposed by Kapur and Rokhlin~\cite{kapur1997high}. The resulting linear system is notoriously difficult to solve iteratively.

{Again, we constructed a HODLR approximation using the proxy surface technique (see, e.g., \cite[Chapter~17]{martinsson2019fast}) on a CPU and copied the HODLR matrix to the GPU.}
 We applied our HODLR solver and the block-sparse solver, and the results are shown in \cref{t:bie_ht}. The factorization time and the solution time are plotted in \cref{f:bie_ht}, and  \cref{f:flops} shows the floating point operations per second (Flop/s).  Due to the oscillatory nature of the fundamental solution of the Helmholtz operator, the numerical ranks from low-rank compressions are higher than those in \cref{s:laplace} associated with the Laplace operator for the same accuracy. As a result, the costs for solving \cref{e:helmholtz_int}  are generally higher.  However, the observations are similar to before: (1) the costs of our GPU solver scaled nearly linearly; (2) our GPU solver achieved significant speedups over the parallel block-sparse solver (the factorization of the latter is faster than that of the sequential method); (3) we obtained a fast direct solver with highly accurate low-rank compressions as shown in \cref{t:bie_ht_acc10}. When relatively low-accuracy compressions were employed, we computed a reasonably accurate preconditioner much faster using much less memory.

\begin{figure}
\centering
\begin{subfigure}{0.24\textwidth}
\centering
\includegraphics[width=\textwidth]{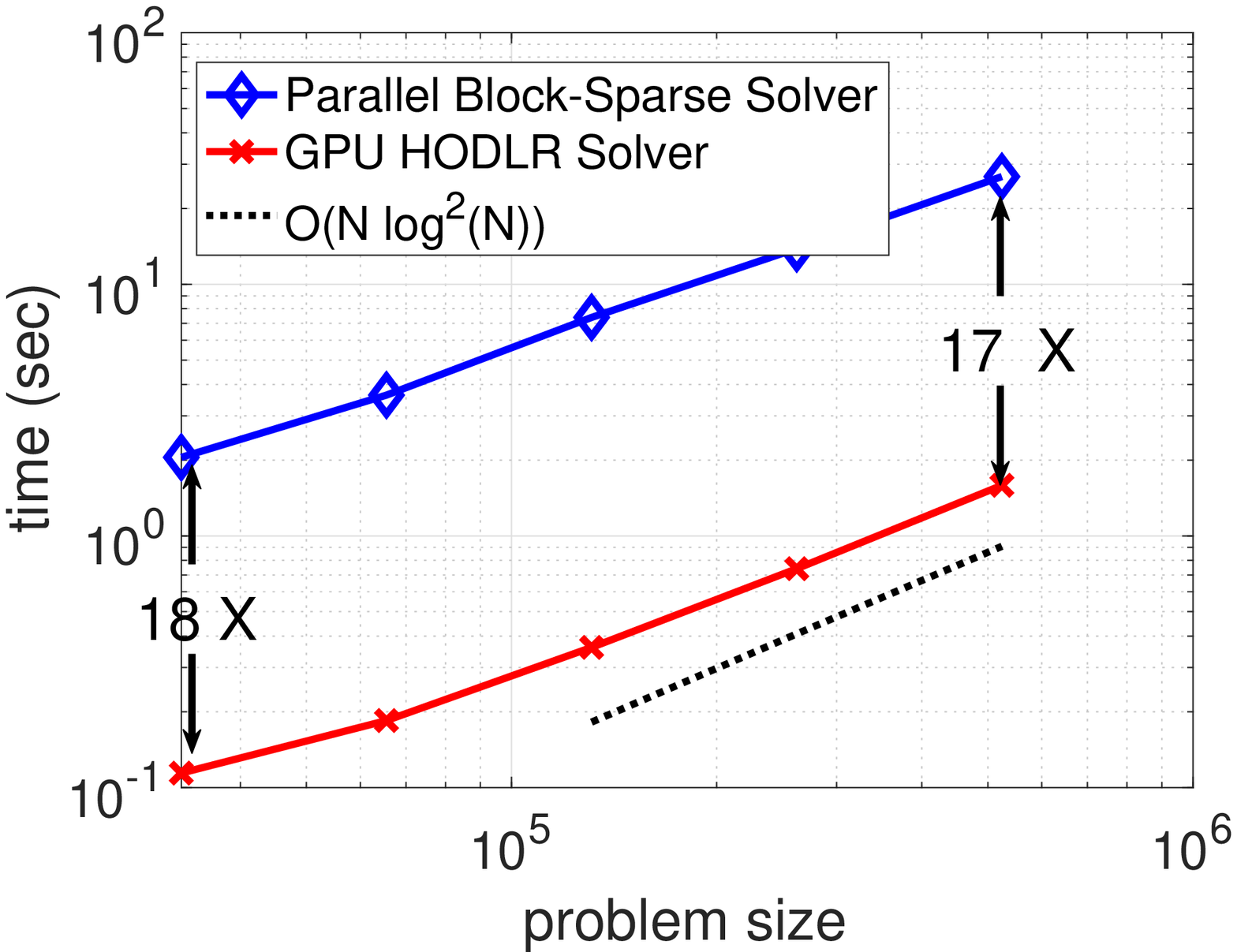}
\caption{High-accuracy factorization}
\end{subfigure}
\begin{subfigure}{0.24\textwidth}
\centering
\includegraphics[width=\textwidth]{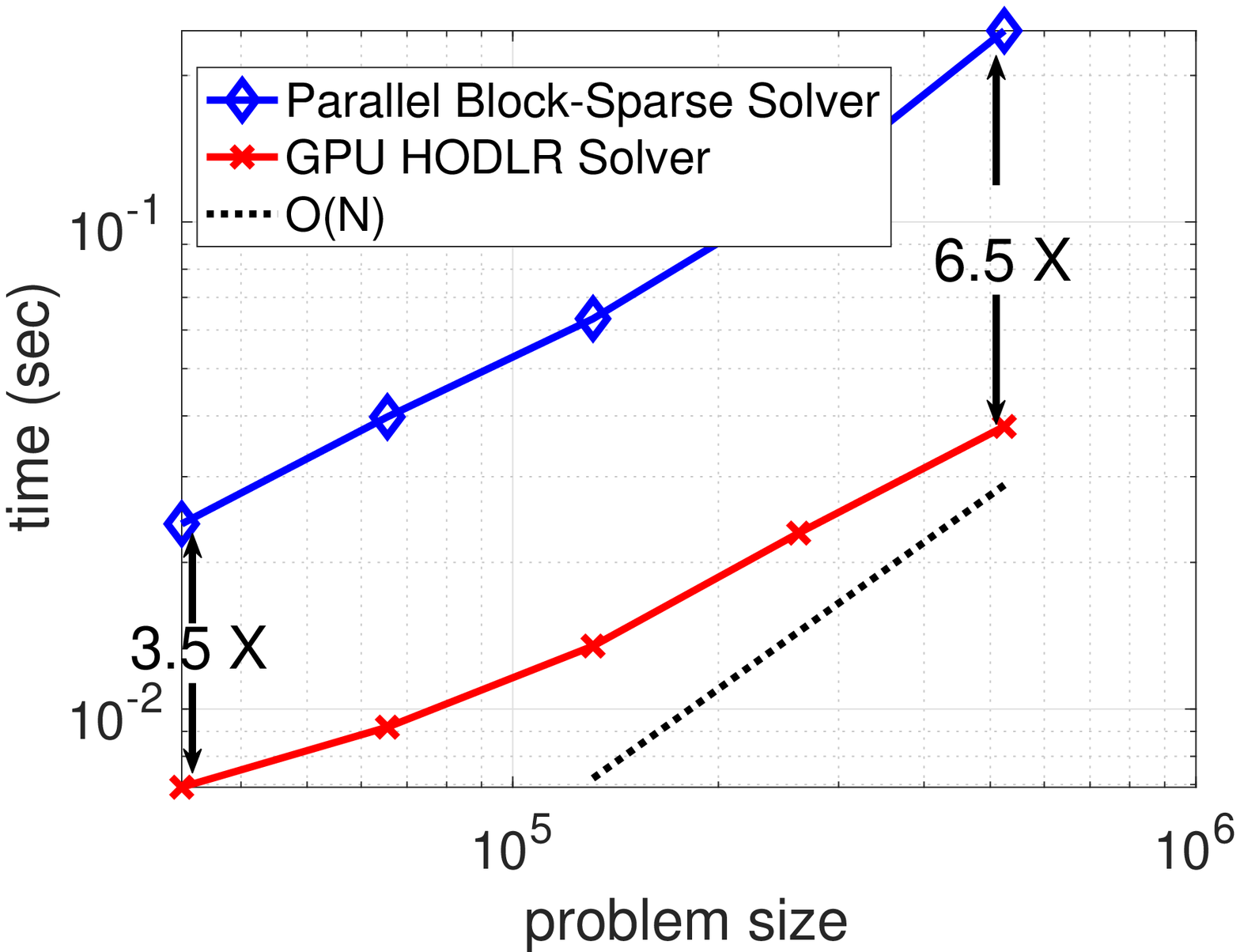}
\caption{High-accuracy solution}
\end{subfigure}
\begin{subfigure}{0.24\textwidth}
\centering
\includegraphics[width=\textwidth]{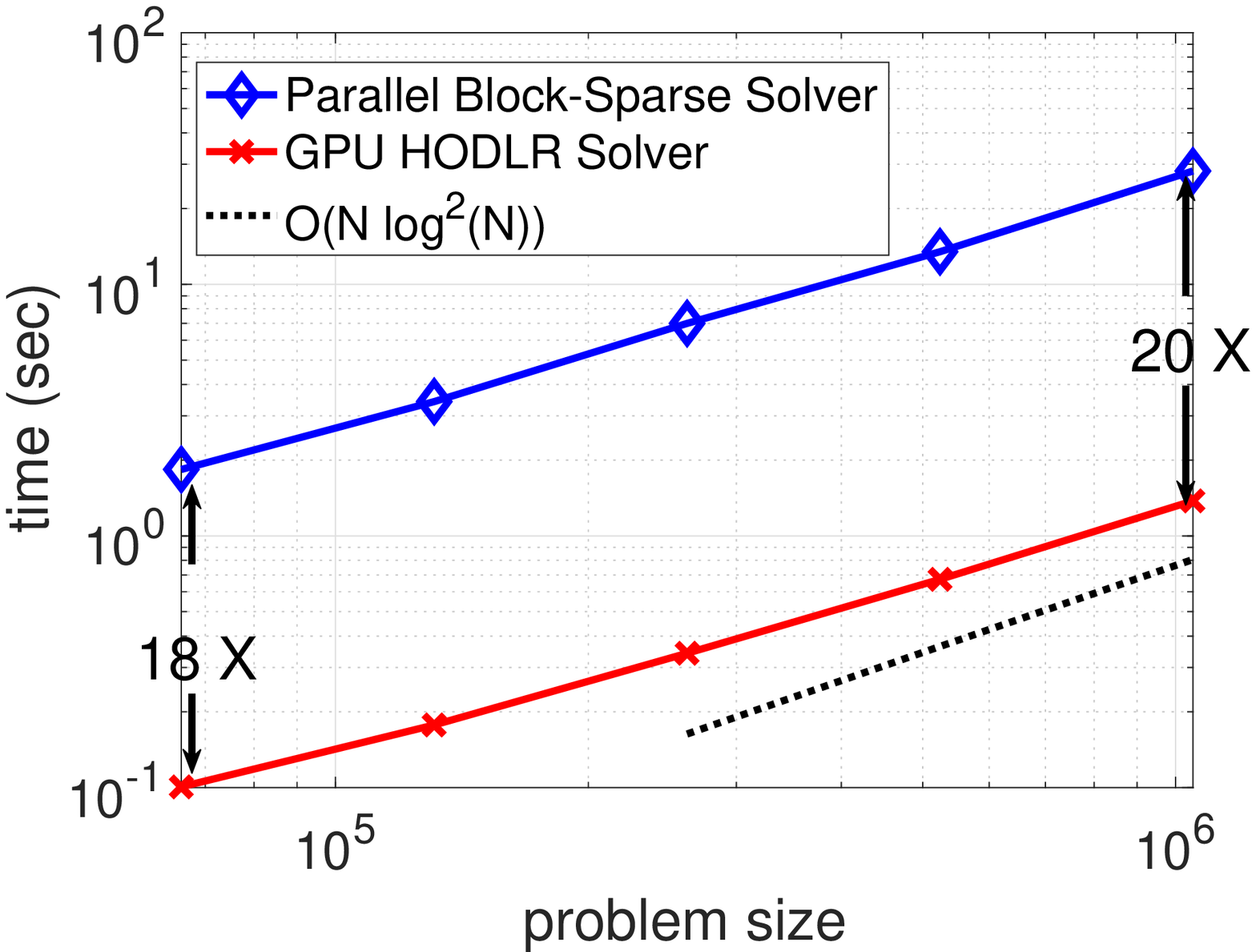}
\caption{Low-accuracy factorization}
\end{subfigure}
\begin{subfigure}{0.24\textwidth}
\centering
\includegraphics[width=\textwidth]{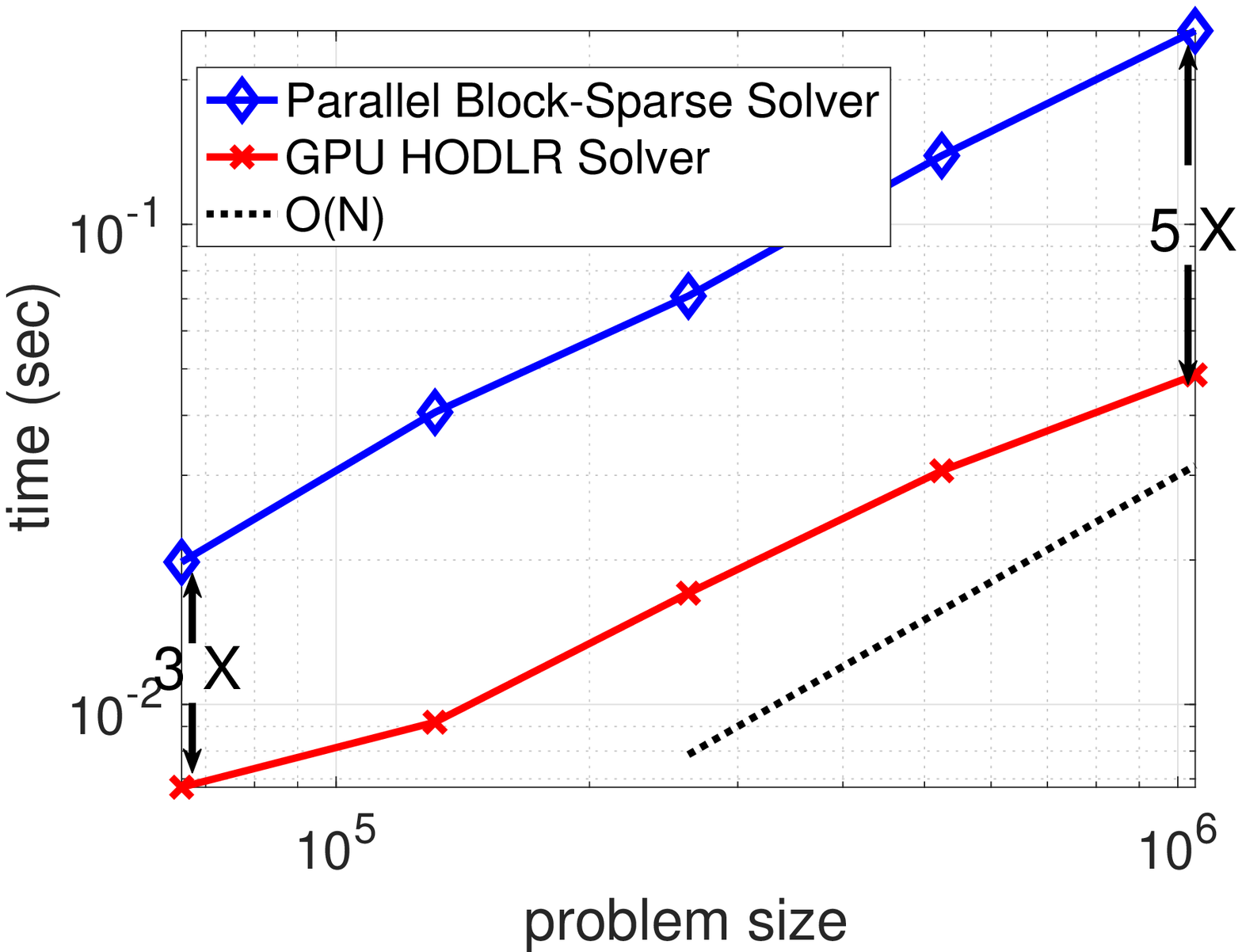}
\caption{Low-accuracy solution}
\end{subfigure}
\caption{Speedups for solving \cref{e:helmholtz_int} discretized with a 6-th order quadrature.}
\label{f:bie_ht}
\end{figure}

\begin{figure}
\centering
\begin{subfigure}{0.24\textwidth}
\centering
\includegraphics[width=\textwidth]{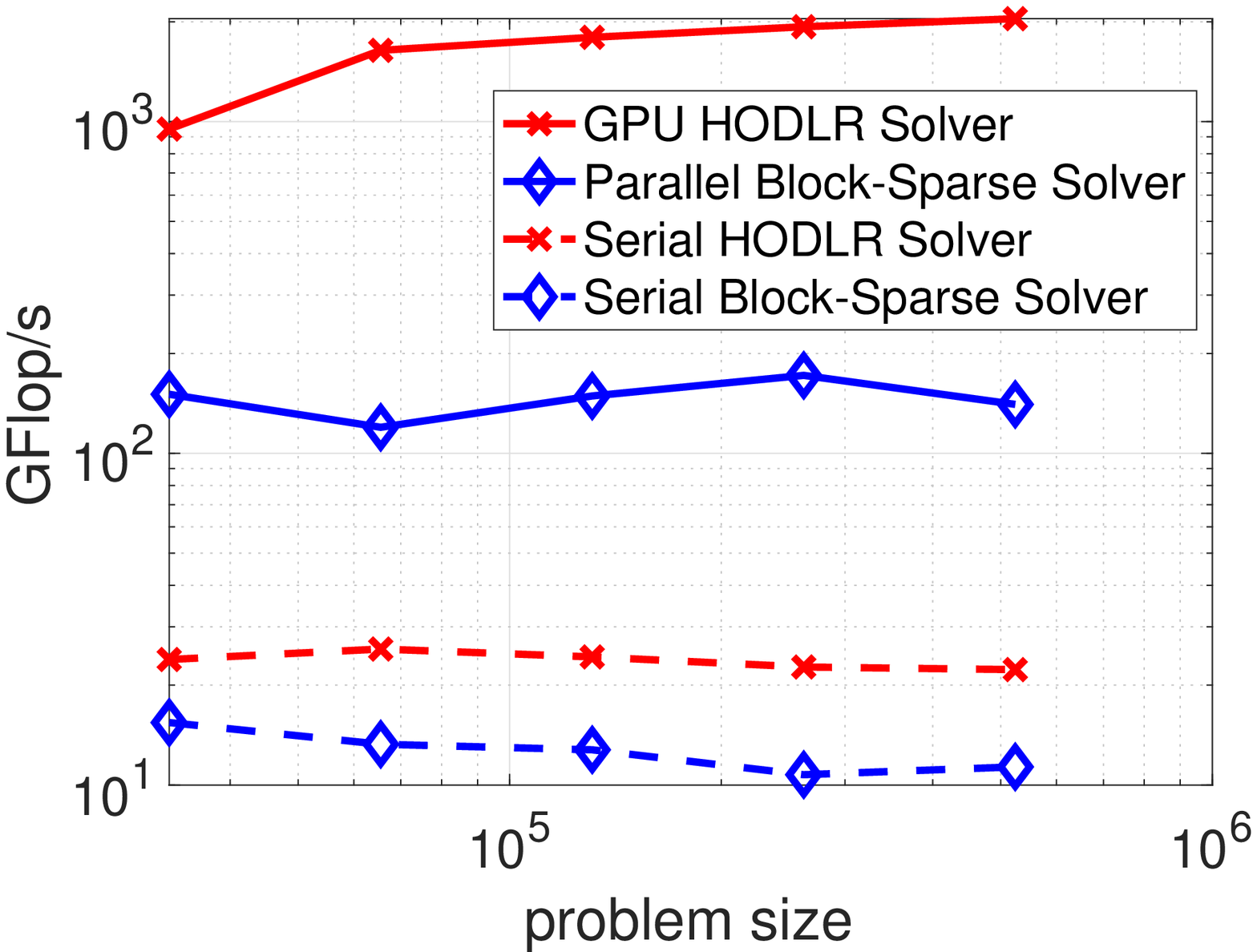}
\caption{High-accuracy factorization}
\label{f:a}
\end{subfigure}
\begin{subfigure}{0.24\textwidth}
\centering
\includegraphics[width=\textwidth]{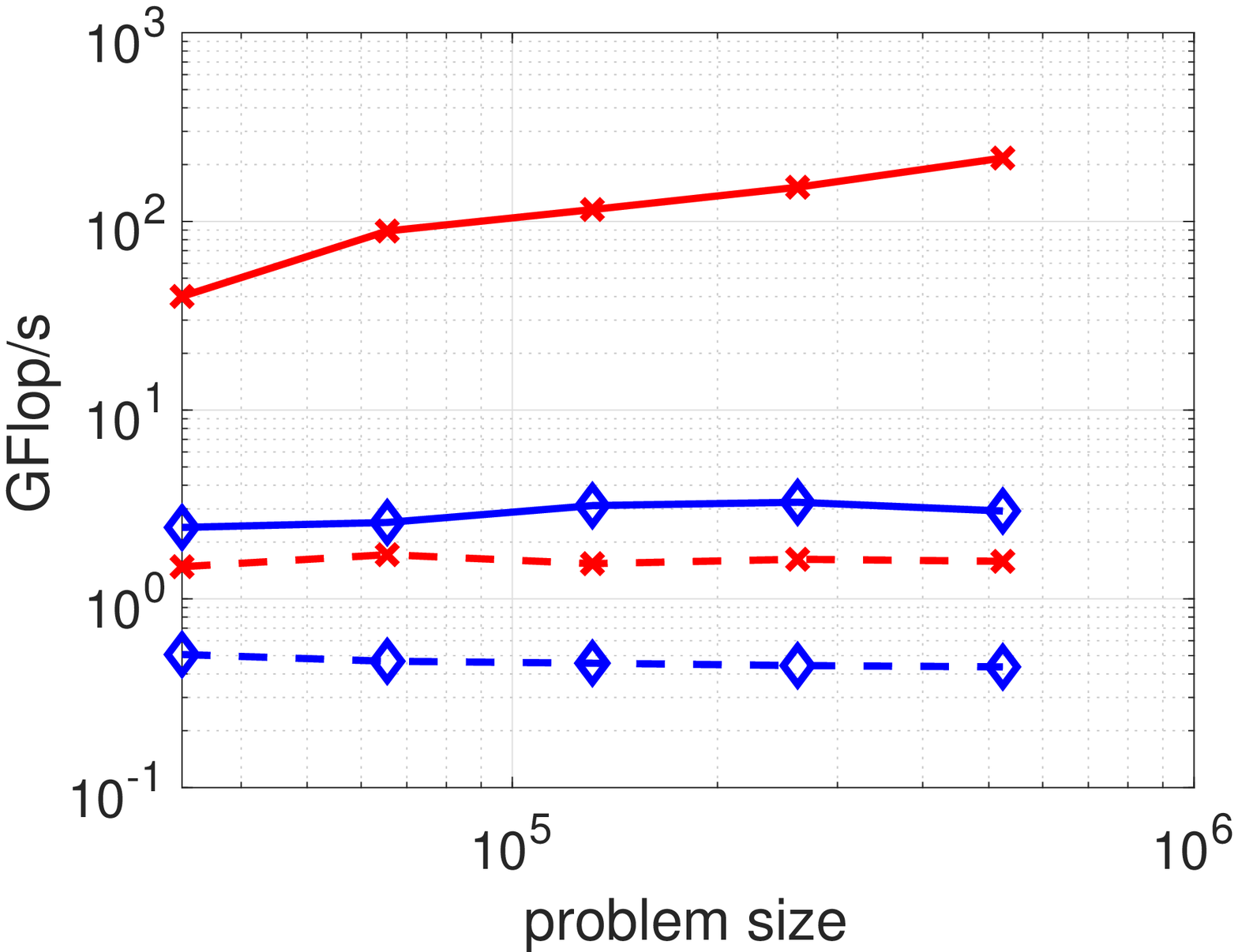}
\caption{High-accuracy solution}
\end{subfigure}
\caption{GFlop/s for solving \cref{e:helmholtz_int}.
For the block-sparse solver,  only the numerical factorization phase is considered in \cref{f:a}.
}
\label{f:flops}
\end{figure}

\begin{table*}
    \caption{\em Results for solving \cref{e:helmholtz_int} with {$\eta = \kappa = 100$}, discretized with a 6-th order quadrature.
    }
    \label{t:bie_ht}
    \begin{subtable}[h]{\textwidth}
    \centering
    \begin{tabular}{c|ccc|ccc|ccc|ccc|c} \toprule
	\multirow{2}{*}{$N$} & \multicolumn{3}{c|}{Serial HODLR Solver} & \multicolumn{3}{c|}{Serial Block-Sparse Solver} & \multicolumn{3}{c|}{Parallel Block-Sparse Solver} & \multicolumn{3}{c|}{GPU HODLR Solver} & \multirow{2}{*}{relres}  \\
	& $t_f$ & $t_s$ & mem & $t_f$ & $t_s$ & mem & $t_f$ & $t_s$ & mem & $t_f$ & $t_s$ & mem & \\ \midrule
	\rowcolor{Gray}
	$2^{15} = 32,768$ & 4.53e+0 & 1.92e-1 & 0.81 & 4.08e+0 & 8.28e-2 & 0.28 & 2.05e+0 & 2.40e-2 & 1.21 & 1.14e-1 & 6.91e-3 & 0.81 & 2.02e-9  \\
	$2^{16} = 65,536$ & 1.18e+1 & 4.86e-1 & 1.70 & 6.40e+0 & 1.51e-1 & 0.51 & 3.63e+0 & 3.98e-2 & 2.08 & 1.85e-1 & 9.18e-3 & 1.70 & 1.34e-9  \\
	\rowcolor{Gray}
	$2^{17} = 131,072$ & 2.66e+1 & 1.01e-0 & 3.58 & 1.10e+1 & 2.87e-1 & 0.96 & 7.39e+0 & 6.33e-2 & 3.97 & 3.61e-1 & 1.35e-2 & 3.58 & 1.67e-9  \\
	$2^{18} = 262,144$ & 6.31e+1 & 2.15e-0 & 7.48 & 1.99e+1 & 5.69e-1 & 1.84 & 1.39e+1 & 1.14e-1 & 7.55 & 7.42e-1 & 2.29e-2 & 7.48 & 7.23e-10  \\
	\rowcolor{Gray}
	$2^{19} = 524,288$ & 1.45e+2 & 5.22e-0 & 15.7 & 3.75e+1 & 1.12e-0 & 3.59 & 2.68e+1 & 2.47e-1 & 14.7 & 1.59e-0 & 3.80e-2 & 15.7 & 1.02e-9 \\
    \bottomrule
    \end{tabular}
    \caption{\em  High-accuracy fast direct solver}
    \label{t:bie_ht_acc10}
    \end{subtable} \\
    \begin{subtable}[h]{\textwidth}
    \centering
    \begin{tabular}{c|ccc|ccc|ccc|ccc|c} \toprule
	\multirow{2}{*}{$N$} & \multicolumn{3}{c|}{Serial HODLR Solver} & \multicolumn{3}{c|}{Serial Block-Sparse Solver} & \multicolumn{3}{c|}{Parallel Block-Sparse Solver} & \multicolumn{3}{c|}{GPU HODLR Solver} & \multirow{2}{*}{relres}   \\
	& $t_f$ & $t_s$ & mem & $t_f$ & $t_s$ & mem & $t_f$ & $t_s$ & mem & $t_f$ & $t_s$ & mem & \\ \midrule
	\rowcolor{Gray}
	$2^{15} = 32,768$ & 2.94e+0 & 1.57e-1 & 0.58 & 1.74e+0 & 4.31e-2 & 0.14 & 1.12e+0 & 1.73e-2 & 0.76 & 6.24e-2 & 4.44e-3 & 0.58 & 1.25e-4  \\
	$2^{16} = 65,536$ & 6.63e+0 & 3.24e-1& 1.17 & 2.61e+0 & 8.72e-2 & 0.25 & 1.84e+0 & 1.98e-2 & 1.24 & 1.00e-1 & 6.73e-3 & 1.17 & 1.98e-4  \\
	\rowcolor{Gray}
	$2^{17} = 131,072$ & 1.51e+1 & 6.71e-1 & 2.37 & 4.10e+0 & 1.37e-1 & 0.45 & 3.42e+0 & 4.06e-2 & 2.37 & 1.77e-1 & 9.19e-3 & 2.37 & 3.04e-4  \\
	$2^{18} = 262,144$ & 3.45e+1 & 1.51e-0 & 4.83 & 7.42e+0 & 2.69e-1 & 0.86 & 6.99e+0 & 7.09e-2 & 4.44 & 3.42e-1 & 1.71e-2 & 4.83 & 3.62e-4  \\
	\rowcolor{Gray}
	$2^{19} = 524,288$ & 7.76e+1 & 3.19e-0 & 9.83 & 1.41e+1 & 5.24e-1 & 1.68 & 1.35e+1 & 1.39e-1 & 8.94 & 6.72e-1 & 3.07e-2 & 9.83 & 3.99e-4  \\
	$2^{20} = 1,048,576$ & 1.75e+2 & 6.92e-0 & 19.8 & 2.72e+1 & 1.05e-0 & 3.32 & 2.82e+1 & 2.53e-1 & 17.5 & 1.38e-0 & 4.86e-2 & 19.8 & 7.21e-4  \\
    \bottomrule
    \end{tabular}
    \caption{\em Low-accuracy robust preconditioner}
    \label{t:bie_ht_acc4}
    \end{subtable}
\end{table*}

%
%


\section{Generalizations and Conclusions}

We introduce algorithms for factorizing a HODLR matrix and for solving linear systems using the factorization on a GPU. The algorithms are based on a new data structure of storing the HODLR matrix and leverage efficient batched dense linear algebra kernels. We benchmarked the performance of our codes on kernel matrices and discretized BIEs, for which we constructed fast solvers with different levels of accuracies. In our numerical experiments, our GPU solver achieved significant speedups against reference methods.

\appendix

\section*{{Off-diagonal ranks in HODLR approximations}} \label{s:rank}

Below are the ranks of off-diagonal blocks from level 1 to the leaf level:

\begin{itemize}

\item
\cref{t:hodlrlib}, $N=2^{21}$ (15 tree levels):
\begin{center}
56 54 45 52 44 30 41 38 38 25 33 24 22 19 18.
\end{center}


\item
\cref{t:bie_dr_acc10}, $N=2^{22}$ (16 tree levels):
\begin{center}
24    22    15    14    13    13    13    13    14    14    15    16    16    17    17    18.
\end{center}

\item
\cref{t:bie_dr_acc6},  $N=2^{24}$ (18 tree levels):
\begin{center}
  1     1     1     2     3     3     4     4     5     5     6     7     7     8     8     9    10    11.
\end{center}

\item
\cref{t:bie_ht_acc10},  $N=2^{19}$ (13 tree levels):
\begin{center}
225   134    97    69    54    46    41    39    37    35    33    31    29.
\end{center}

\item
\cref{t:bie_ht_acc4},  $N=2^{20}$ (14 tree levels):
\begin{center}
166 92 63 39 28 22 19 17 17 17 17 17 17 17.
\end{center}

\end{itemize}

\bibliographystyle{IEEEtran}
\bibliography{biblio}

\begin{thebibliography}{10}
\providecommand{\url}[1]{#1}
\csname url@samestyle\endcsname
\providecommand{\newblock}{\relax}
\providecommand{\bibinfo}[2]{#2}
\providecommand{\BIBentrySTDinterwordspacing}{\spaceskip=0pt\relax}
\providecommand{\BIBentryALTinterwordstretchfactor}{4}
\providecommand{\BIBentryALTinterwordspacing}{\spaceskip=\fontdimen2\font plus
\BIBentryALTinterwordstretchfactor\fontdimen3\font minus
  \fontdimen4\font\relax}
\providecommand{\BIBforeignlanguage}[2]{{%
\expandafter\ifx\csname l@#1\endcsname\relax
\typeout{** WARNING: IEEEtran.bst: No hyphenation pattern has been}%
\typeout{** loaded for the language `#1'. Using the pattern for}%
\typeout{** the default language instead.}%
\else
\language=\csname l@#1\endcsname
\fi
#2}}
\providecommand{\BIBdecl}{\relax}
\BIBdecl

\bibitem{ambikasaran2013mathcal}
S.~Ambikasaran and E.~Darve, ``An {O (N log N)} fast direct solver for partial
  hierarchically semi-separable matrices,'' \emph{Journal of Scientific
  Computing}, vol.~57, no.~3, pp. 477--501, 2013.

\bibitem{aminfar2016fast}
A.~Aminfar, S.~Ambikasaran, and E.~Darve, ``A fast block low-rank dense solver
  with applications to finite-element matrices,'' \emph{Journal of
  Computational Physics}, vol. 304, pp. 170--188, 2016.

\bibitem{gray2001n}
A.~G. Gray and A.~W. Moore, ``N-{B}ody problems in statistical learning,''
  \emph{Advances in neural information processing systems}, pp. 521--527, 2001.

\bibitem{hofmann2008kernel}
T.~Hofmann, B.~Sch{\"o}lkopf, and A.~J. Smola, ``Kernel methods in machine
  learning,'' \emph{The annals of statistics}, vol.~36, no.~3, pp. 1171--1220,
  2008.

\bibitem{ambikasaran2013fast}
S.~Ambikasaran, \emph{Fast algorithms for dense numerical linear algebra and
  applications}.\hskip 1em plus 0.5em minus 0.4em\relax Stanford University,
  2013.

\bibitem{li2014kalman}
J.~Y. Li, S.~Ambikasaran, E.~F. Darve, and P.~K. Kitanidis, ``A kalman filter
  powered by-matrices for quasi-continuous data assimilation problems,''
  \emph{Water Resources Research}, vol.~50, no.~5, pp. 3734--3749, 2014.

\bibitem{ambikasaran2014fast}
S.~Ambikasaran, M.~O'Neil, and K.~R. Singh, ``Fast symmetric factorization of
  hierarchical matrices with applications,'' \emph{arXiv preprint
  arXiv:1405.0223}, 2014.

\bibitem{ambikasaran2015fast}
S.~Ambikasaran, D.~Foreman-Mackey, L.~Greengard, D.~W. Hogg, and M.~O’Neil,
  ``Fast direct methods for {G}aussian processes,'' \emph{IEEE transactions on
  pattern analysis and machine intelligence}, vol.~38, no.~2, pp. 252--265,
  2015.

\bibitem{martinsson2019fast}
P.-G. Martinsson, \emph{Fast direct solvers for elliptic PDEs}.\hskip 1em plus
  0.5em minus 0.4em\relax SIAM, 2019.

\bibitem{davis2016survey}
T.~A. Davis, S.~Rajamanickam, and W.~M. Sid-Lakhdar, ``A survey of direct
  methods for sparse linear systems,'' \emph{Acta Numerica}, vol.~25, pp.
  383--566, 2016.

\bibitem{xia2010superfast}
J.~Xia, S.~Chandrasekaran, M.~Gu, and X.~S. Li, ``Superfast multifrontal method
  for large structured linear systems of equations,'' \emph{SIAM Journal on
  Matrix Analysis and Applications}, vol.~31, no.~3, pp. 1382--1411, 2010.

\bibitem{10.1007/s10915-008-9240-6}
\BIBentryALTinterwordspacing
P.-G. Martinsson, ``A fast direct solver for a class of elliptic partial
  differential equations,'' \emph{J. Sci. Comput.}, vol.~38, no.~3, p.
  316–330, mar 2009. [Online]. Available:
  \url{https://doi.org/10.1007/s10915-008-9240-6}
\BIBentrySTDinterwordspacing

\bibitem{ltaief2021meeting}
H.~Ltaief, J.~Cranney, D.~Gratadour, Y.~Hong, L.~Gatineau, and D.~Keyes,
  ``Meeting the real-time challenges of ground-based telescopes using low-rank
  matrix computations,'' in \emph{Proceedings of the International Conference
  for High Performance Computing, Networking, Storage and Analysis}, 2021, pp.
  1--16.

\bibitem{amestoy2019performance}
P.~R. Amestoy, A.~Buttari, J.-Y. L'excellent, and T.~Mary, ``Performance and
  scalability of the block low-rank multifrontal factorization on multicore
  architectures,'' \emph{ACM Transactions on Mathematical Software (TOMS)},
  vol.~45, no.~1, pp. 1--26, 2019.

\bibitem{akbudak2017tile}
K.~Akbudak, H.~Ltaief, A.~Mikhalev, and D.~Keyes, ``Tile low rank cholesky
  factorization for climate/weather modeling applications on manycore
  architectures,'' in \emph{International Supercomputing Conference}.\hskip 1em
  plus 0.5em minus 0.4em\relax Springer, 2017, pp. 22--40.

\bibitem{al2020solving}
N.~Al-Harthi, R.~Alomairy, K.~Akbudak, R.~Chen, H.~Ltaief, H.~Bagci, and
  D.~Keyes, ``Solving acoustic boundary integral equations using high
  performance tile low-rank lu factorization,'' in \emph{International
  Conference on High Performance Computing}.\hskip 1em plus 0.5em minus
  0.4em\relax Springer, 2020, pp. 209--229.

\bibitem{boukaram2021h2opus}
W.~Boukaram, S.~Zampini, G.~Turkiyyah, and D.~Keyes, ``H2opus-tlr: High
  performance tile low rank symmetric factorizations using adaptive randomized
  approximation,'' \emph{arXiv preprint arXiv:2108.11932}, 2021.

\bibitem{hackbusch1999sparse}
W.~Hackbusch, ``A sparse matrix arithmetic based on {H}-matrices. part {I}:
  Introduction to {H}-matrices,'' \emph{Computing}, vol.~62, no.~2, pp.
  89--108, 1999.

\bibitem{5555}
W.~Hackbusch and B.~N. Khoromskij, ``A sparse {H}-matrix arithmetic. part ii:
  Application to multi-dimensional problems,'' \emph{Computing}, vol.~64,
  no.~1, p. 21–47, Jan. 2000.

\bibitem{hackbusch2002data}
W.~Hackbusch and S.~B{\"o}rm, ``Data-sparse approximation by adaptive
  {$H^2$}-matrices,'' \emph{Computing}, vol.~69, no.~1, pp. 1--35, 2002.

\bibitem{martinsson2005fast}
P.-G. Martinsson and V.~Rokhlin, ``A fast direct solver for boundary integral
  equations in two dimensions,'' \emph{Journal of Computational Physics}, vol.
  205, no.~1, pp. 1--23, 2005.

\bibitem{chandrasekaran2006fast}
S.~Chandrasekaran, M.~Gu, and T.~Pals, ``A fast {ULV} decomposition solver for
  hierarchically semiseparable representations,'' \emph{SIAM Journal on Matrix
  Analysis and Applications}, vol.~28, no.~3, pp. 603--622, 2006.

\bibitem{xia2010fast}
J.~Xia, S.~Chandrasekaran, M.~Gu, and X.~S. Li, ``Fast algorithms for
  hierarchically semiseparable matrices,'' \emph{Numerical Linear Algebra with
  Applications}, vol.~17, no.~6, pp. 953--976, 2010.

\bibitem{grasedyck2009domain}
L.~Grasedyck, R.~Kriemann, and S.~Le~Borne, ``Domain decomposition based {H-LU}
  preconditioning,'' \emph{Numerische Mathematik}, vol. 112, no.~4, pp.
  565--600, 2009.

\bibitem{kriemann2013fancyscript}
R.~Kriemann, ``{H-LU} factorization on many-core systems,'' \emph{Computing and
  Visualization in Science}, vol.~16, no.~3, pp. 105--117, 2013.

\bibitem{ho2016hierarchical}
K.~L. Ho and L.~Ying, ``Hierarchical interpolative factorization for elliptic
  operators: integral equations,'' \emph{Comm. Pure Appl. Math}, vol.~69,
  no.~7, pp. 1314--1353, 2016.

\bibitem{rouet2016distributed}
F.-H. Rouet, X.~S. Li, P.~Ghysels, and A.~Napov, ``A distributed-memory package
  for dense hierarchically semi-separable matrix computations using
  randomization,'' \emph{ACM Transactions on Mathematical Software (TOMS)},
  vol.~42, no.~4, pp. 1--35, 2016.

\bibitem{coulier2017inverse}
P.~Coulier, H.~Pouransari, and E.~Darve, ``The inverse fast multipole method:
  using a fast approximate direct solver as a preconditioner for dense linear
  systems,'' \emph{SIAM Journal on Scientific Computing}, vol.~39, no.~3, pp.
  A761--A796, 2017.

\bibitem{minden2017recursive}
V.~Minden, K.~L. Ho, A.~Damle, and L.~Ying, ``A recursive skeletonization
  factorization based on strong admissibility,'' \emph{Multiscale Modeling \&
  Simulation}, vol.~15, no.~2, pp. 768--796, 2017.

\bibitem{xia2021multi}
J.~Xia, ``Multi-layer hierarchical structures,'' \emph{CSIAM Transaction of
  Applied Mathematics}, vol.~2, pp. 263--296, 2021.

\bibitem{liu2021sparse}
Y.~Liu, P.~Ghysels, L.~Claus, and X.~S. Li, ``Sparse approximate multifrontal
  factorization with butterfly compression for high-frequency wave equations,''
  \emph{SIAM Journal on Scientific Computing}, vol.~43, no.~5, pp. S367--S391,
  2021.

\bibitem{sushnikova2022fmm}
D.~Sushnikova, L.~Greengard, M.~O'Neil, and M.~Rachh, ``{FMM-LU}: A fast direct
  solver for multiscale boundary integral equations in three dimensions,''
  \emph{arXiv preprint arXiv:2201.07325}, 2022.

\bibitem{ambikasaran2019hodlrlib}
S.~Ambikasaran, K.~R. Singh, and S.~S. Sankaran, ``{HODLR}lib: A library for
  hierarchical matrices,'' \emph{Journal of Open Source Software}, vol.~4,
  no.~34, p. 1167, 2019.

\bibitem{massei2020hm}
S.~Massei, L.~Robol, and D.~Kressner, ``hm-toolbox: Matlab software for {HODLR}
  and {HSS} matrices,'' \emph{SIAM Journal on Scientific Computing}, vol.~42,
  no.~2, pp. C43--C68, 2020.

\bibitem{dong2021simpler}
Y.~Dong and P.-G. Martinsson, ``Simpler is better: A comparative study of
  randomized algorithms for computing the cur decomposition,'' \emph{arXiv
  preprint arXiv:2104.05877}, 2021.

\bibitem{lin2011fast}
L.~Lin, J.~Lu, and L.~Ying, ``Fast construction of hierarchical matrix
  representation from matrix--vector multiplication,'' \emph{Journal of
  Computational Physics}, vol. 230, no.~10, pp. 4071--4087, 2011.

\bibitem{martinsson2016compressing}
P.-G. Martinsson, ``Compressing rank-structured matrices via randomized
  sampling,'' \emph{SIAM Journal on Scientific Computing}, vol.~38, no.~4, pp.
  A1959--A1986, 2016.

\bibitem{levitt2022linear}
J.~Levitt and P.-G. Martinsson, ``Linear-complexity black-box randomized
  compression of hierarchically block separable matrices,'' \emph{arXiv
  preprint arXiv:2205.02990}, 2022.

\bibitem{fernando2017scalable}
I.~D. Fernando, S.~Jayasena, M.~Fernando, and H.~Sundar, ``A scalable
  hierarchical semi-separable library for heterogeneous clusters,'' in
  \emph{2017 46th International Conference on Parallel Processing
  (ICPP)}.\hskip 1em plus 0.5em minus 0.4em\relax IEEE, 2017, pp. 513--522.

\bibitem{martinsson2011fast}
P.-G. Martinsson, ``A fast randomized algorithm for computing a hierarchically
  semiseparable representation of a matrix,'' \emph{SIAM Journal on Matrix
  Analysis and Applications}, vol.~32, no.~4, pp. 1251--1274, 2011.

\bibitem{boukaram2019randomized}
W.~Boukaram, G.~Turkiyyah, and D.~Keyes, ``Randomized {GPU} algorithms for the
  construction of hierarchical matrices from matrix-vector operations,''
  \emph{SIAM Journal on Scientific Computing}, vol.~41, no.~4, pp. C339--C366,
  2019.

\bibitem{yu2017geometry}
C.~D. Yu, J.~Levitt, S.~Reiz, and G.~Biros, ``Geometry-oblivious {FMM} for
  compressing dense {SPD} matrices,'' in \emph{Proceedings of the International
  Conference for High Performance Computing, Networking, Storage and Analysis},
  2017, pp. 1--14.

\bibitem{chenhan2018distributed}
D.~Y. Chenhan, S.~Reiz, and G.~Biros, ``Distributed-memory hierarchical
  compression of dense {SPD} matrices,'' in \emph{SC18: International
  Conference for High Performance Computing, Networking, Storage and
  Analysis}.\hskip 1em plus 0.5em minus 0.4em\relax IEEE, 2018, pp. 183--197.

\bibitem{chandra2001parallel}
R.~Chandra, L.~Dagum, D.~Kohr, R.~Menon, D.~Maydan, and J.~McDonald,
  \emph{Parallel programming in OpenMP}.\hskip 1em plus 0.5em minus 0.4em\relax
  Morgan kaufmann, 2001.

\bibitem{ho2012fast}
K.~L. Ho and L.~Greengard, ``A fast direct solver for structured linear systems
  by recursive skeletonization,'' \emph{SIAM Journal on Scientific Computing},
  vol.~34, no.~5, pp. A2507--A2532, 2012.

\bibitem{ambikasaran2014inverse}
S.~Ambikasaran and E.~Darve, ``The inverse fast multipole method,'' \emph{arXiv
  preprint arXiv:1407.1572}, 2014.

\bibitem{takahashi2020parallelization}
T.~Takahashi, C.~Chen, and E.~Darve, ``Parallelization of the inverse fast
  multipole method with an application to boundary element method,''
  \emph{Computer Physics Communications}, vol. 247, p. 106975, 2020.

\bibitem{pouransari2017fast}
H.~Pouransari, P.~Coulier, and E.~Darve, ``Fast hierarchical solvers for sparse
  matrices using extended sparsification and low-rank approximation,''
  \emph{SIAM Journal on Scientific Computing}, vol.~39, no.~3, pp. A797--A830,
  2017.

\bibitem{sushnikova2018compress}
D.~A. Sushnikova and I.~V. Oseledets, ````compress and eliminate” solver for
  symmetric positive definite sparse matrices,'' \emph{SIAM Journal on
  Scientific Computing}, vol.~40, no.~3, pp. A1742--A1762, 2018.

\bibitem{chen2018distributed}
C.~Chen, H.~Pouransari, S.~Rajamanickam, E.~G. Boman, and E.~Darve, ``A
  distributed-memory hierarchical solver for general sparse linear systems,''
  \emph{Parallel Computing}, vol.~74, pp. 49--64, 2018.

\bibitem{chen2019robust}
C.~Chen, L.~Cambier, E.~G. Boman, S.~Rajamanickam, R.~S. Tuminaro, and
  E.~Darve, ``A robust hierarchical solver for ill-conditioned systems with
  applications to ice sheet modeling,'' \emph{Journal of Computational
  Physics}, vol. 396, pp. 819--836, 2019.

\bibitem{rotne1969variational}
J.~Rotne and S.~Prager, ``Variational treatment of hydrodynamic interaction in
  polymers,'' \emph{The Journal of Chemical Physics}, vol.~50, no.~11, pp.
  4831--4837, 1969.

\bibitem{yamakawa1970transport}
H.~Yamakawa, ``Transport properties of polymer chains in dilute solution:
  hydrodynamic interaction,'' \emph{The Journal of Chemical Physics}, vol.~53,
  no.~1, pp. 436--443, 1970.

\bibitem{kress1989linear}
R.~Kress, V.~Maz'ya, and V.~Kozlov, \emph{Linear integral equations}.\hskip 1em
  plus 0.5em minus 0.4em\relax Springer, 1989, vol.~82.

\bibitem{mclean2000strongly}
W.~McLean and W.~C.~H. McLean, \emph{Strongly elliptic systems and boundary
  integral equations}.\hskip 1em plus 0.5em minus 0.4em\relax Cambridge
  university press, 2000.

\bibitem{gillman2012direct}
A.~Gillman, P.~M. Young, and P.-G. Martinsson, ``A direct solver with {O(N)}
  complexity for integral equations on one-dimensional domains,''
  \emph{Frontiers of Mathematics in China}, vol.~7, no.~2, pp. 217--247, 2012.

\bibitem{davis1997unsymmetric}
T.~A. Davis and I.~S. Duff, ``An unsymmetric-pattern multifrontal method for
  sparse lu factorization,'' \emph{SIAM Journal on Matrix Analysis and
  Applications}, vol.~18, no.~1, pp. 140--158, 1997.

\bibitem{davis2004algorithm}
T.~A. Davis, ``Algorithm 832: Umfpack v4. 3---an unsymmetric-pattern
  multifrontal method,'' \emph{ACM Transactions on Mathematical Software
  (TOMS)}, vol.~30, no.~2, pp. 196--199, 2004.

\bibitem{schenk2000efficient}
O.~Schenk, K.~G{\"a}rtner, and W.~Fichtner, ``Efficient sparse lu factorization
  with left-right looking strategy on shared memory multiprocessors,''
  \emph{BIT Numerical Mathematics}, vol.~40, no.~1, pp. 158--176, 2000.

\bibitem{schenk2004solving}
O.~Schenk and K.~G{\"a}rtner, ``Solving unsymmetric sparse systems of linear
  equations with pardiso,'' \emph{Future Generation Computer Systems}, vol.~20,
  no.~3, pp. 475--487, 2004.

\bibitem{kwack2016performance}
J.~Kwack, G.~Bauer, and S.~Koric, ``Performance test of parallel linear
  equation solvers on blue waters--cray xe6/xk7 system,'' in \emph{Preceedings
  of the Cray Users Group Meeting (CUG2016), London, England}, 2016.

\bibitem{swirydowicz2021linear}
K.~{\'S}wirydowicz, E.~Darve, W.~Jones, J.~Maack, S.~Regev, M.~A. Saunders,
  S.~J. Thomas, and S.~Pele{\v{s}}, ``Linear solvers for power grid
  optimization problems: a review of gpu-accelerated linear solvers,''
  \emph{Parallel Computing}, p. 102870, 2021.

\bibitem{kapur1997high}
S.~Kapur and V.~Rokhlin, ``High-order corrected trapezoidal quadrature rules
  for singular functions,'' \emph{SIAM Journal on Numerical Analysis}, vol.~34,
  no.~4, pp. 1331--1356, 1997.

\end{thebibliography}

\end{document}